\documentclass[12pt,a4paper]{amsart}

\usepackage{paralist, mathdots} 

\usepackage{amssymb} 
\usepackage{amsmath,tikz,wasysym,multirow,enumitem,lscape,multirow,subfigure} 
\usepackage{hyperref} 
\usepackage{amsthm}
\usepackage{mathtools}
\usepackage{soul}
\usepackage{cite}
\usepackage{enumitem}

\tikzset{every node/.style={draw, circle, inner sep=2pt},
every picture/.append style={thick,scale=0.8},
every label/.style={draw=none, rectangle}}

\input xy
\xyoption{all}

\usepackage[all]{xy}

\newcommand{\bfb}{{\bf b}}
\newcommand{\bfx}{{\bf x}}
\newcommand{\bfy}{{\bf y}}
\newcommand{\bfv}{{\bf v}}
\newcommand{\calS}{{\mathcal S}}

\newcommand{\bN}{\mathbb{N}}

\newcommand{\Sc}{\overline{{\mathcal S}_0}}
\newcommand{\gssp}{{\mathcal G}^{\text{SSP}}}
\newcommand{\trans}{^\top}
\newcommand{\bR}{{\mathbb R}}

\DeclareMathOperator{\supp}{supp}

\DeclareMathOperator{\vecop}{vec}

 \usepackage{graphicx}
 \usepackage{tikz}
 \usetikzlibrary{positioning}
  \tikzset{main node/.style={circle,fill=red!10,draw,minimum size=0.2cm,inner sep=0pt},
             }

\newtheorem{theorem}{Theorem}[section]

\newtheorem{lemma}[theorem]{Lemma}

\theoremstyle{definition}
\newtheorem{example}[theorem]{Example}
\newtheorem{definition}[theorem]{Definition}
\newtheorem{remark}[theorem]{Remark}

\newtheorem{question}[theorem]{Question}

\title{The strong spectral property for some families of unicyclic graphs}

\author{Sara Koljan\v{c}i\'{c}}
\address[S. Koljan\v{c}i\'{c}]{Faculty of Natural Sciences and Mathematics, University of Banja Luka, Mladena Stojanovi\' ca 2, 78000 Banja Luka, Bosnia and Herzegovina}
\email{sara.jevdjenic@pmf.unibl.org}

\author{Polona Oblak}
\address[P.~Oblak]{Faculty of Computer and Information Science, University of Ljubljana, Ve\v cna pot 113, 1000 Ljubljana, Slovenia; Faculty of Mathematics and Physics, University of Ljubljana and Institute of Mathematics, Physics, and Mechanics, Jadranska ulica 19, 1000 Ljubljana, Slovenia}
\email{polona.oblak@fri.uni-lj.si}

\thanks{The authors acknowledge the financial support from the bilateral projects no.~BI-BA/19-20-044 and no.~BI-BA-24-25-024 funded by the Slovenian Research Agency and Ministry of Civil Affairs, Bosnia and Herzegovina. Polona Oblak has received funding from the Slovenian Research Agency, research core funding no.~P1-0222.}

\date{\today}

\begin{document}

\maketitle
\begin{abstract}
To find all the possible spectra of all real symmetric matrices whose off-diagonal pattern is prescribed by the adjacencies of a given graph $G$, the Strong Spectral Property turned out to be of crucial importance.
In particular, we investigate the set $\gssp$ of all simple graphs $G$ with the property that each symmetric matrix of the pattern of $G$ has the Strong Spectral Property. 
In this paper, we completely characterize unicyclic graphs of girth three in $\gssp$. We prove that any tadpole graph of girth at most five is in $\gssp$ and we show that the same is not valid for girth six tadpole graphs. 
\end{abstract}  

\bigskip



\bigskip

\section{Introduction}

The importance of matrices that satisfy the Strong Spectral Property is well recognized, particularly in their applications to the Inverse Eigenvalue Problem for a graph (IEP-$G$). For a given connected simple graph $G$ with $n$ vertices, the IEP-$G$ wants to find all possible spectra of symmetric matrices having the same zero-nonzero off-diagonal pattern as the adjacency matrix of $G$. Formally, let $\calS(G)$ denote the set of all real symmetric $n \times n$ matrices $A=[a_{ij}]$ such that for $i\ne j$ an element $a_{ij}$ is nonzero if and only if $\{i, j\}\in E(G)$, with no restriction to the diagonal entries of $A$. To resolve the IEP-$G$ for a graph $G$ we would like to find all multisets of $n$ real numbers that are the spectrum of some matrix $A\in \calS(G)$. IEP-$G$ is a very difficult problem, and the Strong Spectral Property of matrices became an essential tool for addressing it. For further details on this topic, the reader is referred to~\cite{Hogben-Lin-Shader-IEPG-book}.

We say that a symmetric matrix $A$ has the \emph{Strong Spectral Property (SSP)} if the only symmetric matrix $X$ satisfying $$A \circ X= I \circ X=O \ \text{and} \ [A, X]=O,$$ is the zero matrix $X=O$. This definition and its importance to IEP-$G$ was initiated in \cite{Barrett17GeneralizationsSAP}. It was shown in~\cite[Theorem 10]{Barrett17GeneralizationsSAP} that every spectrum of a matrix $A\in \calS(G)$ with the SSP is also the spectrum of a matrix $A'\in \calS(G')$ with the SSP, where $G'$ is any supergraph of $G$ on the same vertex set.  Using this result, the IEP-$G$ for graphs on at most five vertices was completely resolved~\cite[Section~3]{Barrett20Multiplicities}.  In~\cite{Fallat22Bifurcation} the authors showed that if a spectrum can be realized by a matrix $A\in \calS(G)$ with the SSP, then all the nearby spectra can also be realized by matrices in $\calS(G)$ with the SSP. Matrices with the SSP have been broadly studied since then, and with their help, IEP-$G$ for several families of graphs has been resolved, see for example~\cite{2024-Abiad-SNIP,Ahn21Six,allred2023strong,Barrett20Multiplicities,Barrett17GeneralizationsSAP,Bjorkman-2018-q,Fallat22Bifurcation,Lin-21-IEPG-block-graphs,Lin-23-Liberation-Set}, as well as their generalizations to the nonsymmetric case,~\cite{Arav-24-nSSP}. 

Since the SSP of a matrix depends on its entries, it is difficult to find a symmetric matrix with the SSP and a given spectrum. Therefore, the authors in~\cite{Lin20SSPgraph} introduced the combinatorial concept of the graph with the SSP, which is a graph $G$ for which every matrix $A\in \calS(G)$ has the SSP.  Let $$\gssp=\{G\colon \text{every } A\in \calS(G) \text{ has the SSP}\}$$ 
and if $G\in \gssp$, we say that $G$ is an \emph{SSP graph}. For example, the complete graph is in $\gssp$. Note that SSP graphs play a major role in resolving IEP-$G$; if $G\in\gssp$, then by~\cite[Theorem~10]{Barrett17GeneralizationsSAP} all the spectra realized by matrices in $\calS(G)$ can be also realized by matrices of any supergraph of $G$ on the same vertex set. In this sense, it is important to characterize sparse graphs in $\gssp$.

In~\cite{Lin20SSPgraph} the authors showed that several families of graphs are in $\gssp$; we list some of them in Lemma~\ref{lem:known-ssp-graphs}. Moreover, they gave a complete characterization of trees which belong to $\gssp$; a tree is an SSP graph if and only if it does not contain a vertex of degree at least four or two vertices of degree three. They also introduced the notion of a \emph{barbell partition} of a graph, whose existence forbids the graph to be in $\gssp$.  In \cite{allred2023strong}, the authors construct several graphs having the barbell partition hence they are not SSP graphs.  
In particular, they proved that any cactus graph having at least two cycles is not an SSP graph. 
To the best of our knowledge, beside lollipop graphs of girth three and some examples of unicylic graphs of order at most 6, there is no characterization of unicyclic graphs in $\gssp$ of an arbitrary girth.

This paper provides a complete characterization of unicyclic graphs of girth three in $\gssp$, see Theorem~\ref{thm:unicyclic-girth-3-graphs}. This will be the next fully characterized family of sparse graphs in $\gssp$ after the characterization of trees in $\gssp$ given in~\cite[Theorem~4.3]{Lin20SSPgraph}. Moreover, we will prove that any tadpole graph of girth at most five is in $\gssp$. Sections~\ref{sec:girth=3} through~\ref{sec:girth=5} will be dedicated to unicyclic graphs with girths ranging from three to five, respectively. We will show in Section~\ref{sec:girth=6} that the same is not valid for girth six tadpole graphs, and we finish with several open questions.

\section{Related work and notation}

Throughout the paper, let $G=(V(G),E(G))$ denote a simple graph. We will consider primarily unicyclic graphs, i.e. the graphs with the unique cycle. The number of vertices of a unicyclic graph's unique cycle is called the  \emph{girth of the graph}. 

First, we recall the related work and introduce the notation.

\subsection{Notation for graphs and matrices} We  denote a path on $n$ vertices by $P_n$, the cycle on $n$ vertices by $C_n$ and the tadpole graph obtained by joining a cycle $C_m$ to a path $P_n$ with a bridge, by $T_{m,n}$.  For a graph $G$, its complement is denoted by $G^c$. 
If $G$ and $H$ are two graphs with the same set of vertices and $E(G)= E(H) \cup \{e_1,\ldots, e_k\}$, we abbreviate $G=H+\{e_1,\ldots, e_k\}$. If $k=1$, we also write $G=H+e_1$.
 Let $G^{r}$ ($G^{(r)}$, respectively) be the \emph{$r$-th power of the graph $G$} (\emph{$r$-th strong power of the graph $G$}, respectively), which is defined as a graph on the same set of vertices, and when two vertices are adjacent if their distance in $G$ is at most $r$ (equal to $r$, respectively). For example, in Figure~\ref{fig:Z3-strong-power} we present a graph's first three strong powers.

For any $n\in\bN$ let $[n]=\{1,2,\ldots,n\}$. If $\alpha,\beta \subseteq [n]$, then $A[\alpha, \beta]$ is a submatrix of an $n\times n$ matrix $A$ comprised by the elements at the intersection of the rows indexed by $\alpha$ and the columns indexed by $\beta$.  If $\alpha=[n]$ ($\beta=[n]$, respectively), then we abbreviate the notation to $A[\colon,\beta]$ ($A[\alpha,\colon]$, respectively).
We use a similar notation $G[\alpha]$ to denote the \emph{induced subgraph} of $G$ on vertices $\alpha\subseteq V(G)$, i.e. we delete from $G$ all vertices not in $\alpha$.

We use $\bf 0$ to denote the zero vector, $O$ to denote the zero matrix, and $I$ for the identity matrix. To emphasize their sizes, we will sometimes use indices as subscripts. Moreover, for two square matrices of the same size, we denote by $A \circ B$ their Haddamard product and by $[A,B]=AB-BA$ their commutator.
 
\subsection{Strong spectral property of matrices} Let us recall the definition of the strong spectral property of a symmetric matrix.
\begin{definition}\label{def:SSP}
    A symmetric matrix $A$ has the \emph{Strong Spectral Property (SSP)} if the only symmetric matrix $X$ satisfying $$A \circ X= I \circ X=O \ \text{and} \ [A, X]=O$$ is the zero matrix $X=O$. 
\end{definition}
Let us denote by $\Sc(G)$ the set of all 
real symmetric $n\times n$ matrices whose $(i,j)$-entry is zero if $i=j$ or $\{i,j\}\in E(G)$. Note that for $\{i,j\}\notin E(G)$ matrix $M\in \Sc(G)$ can have $M_{i,j}=0$ or $M_{i,j}\ne 0$. 

For $A\in\calS(G)$ the first two constraints $A\circ X=I\circ X=O$ in Definition~\ref{def:SSP} imply that  $X=[x_{i,j}]\in \Sc(G^c)$. 
 The third constraint $[A,X]=O$ can be rewritten as a linear system $\Psi \bfx ={\bf 0}$ where $\bfx$ is a column vector containing the variables $x_{i,j}$, $\{i,j\}\in E(G^c)$, 
and $\Psi$ is the matrix of the system $[A,X]=O$. 
In matrix $\Psi$ we index the rows by the position of the corresponding equation in the system $[A,X]=O$: row $\{i,j\}$ corresponds to the equation $[A, X]_{i,j}=O$. The columns of $\Psi$ will be indexed by pairs $\{k,\ell\}$, where column $\{k,\ell\}$ consists of coefficients corresponding to a variable $x_{k,\ell}$, $\{k,\ell\}\in E(G^c)$. 

Note that the matrix $\Psi$ depends on the order of the variables in $\bfx$ and the order of the equations of $[A,X]=O$. If the columns are ordered in the ascending lexicographic order, we denote the matrix $\Psi$ by $\Psi_S(A)$  and call it the \emph{SSP verification matrix of $A$}. The SSP verification matrix was initially defined in~\cite[Definition~7.1]{Barrett20Multiplicities}, where it was used to prove the Matrix Liberation Lemma and the Augmentation Lemma, demonstrating its effectiveness in the context of IEP-$G$. Note that the authors in~\cite{Barrett17GeneralizationsSAP,Barrett20Multiplicities,Lin-23-Liberation-Set} use the transposed version of the SSP verification matrix compared to the one we use in this paper. 

We will often consider a particular reordering of rows and columns of the verification matrix and note that every such matrix is of the form $\Psi=P \Psi_S(A) Q$, where $P$ and $Q$ are two permutation matrices. By definition, matrix $A$ has the SSP if and only if the SSP verification matrix $\Psi_S(A)$ and all the corresponding matrices $\Psi$ have full column ranks, see also~\cite[Theorem~31]{Barrett17GeneralizationsSAP} and~\cite[Section~7]{Barrett20Multiplicities}.

\subsection{Strong spectral property of graphs} 

In this paper, we work with the graphs $G$ and their properties that guarantee the SSP for all matrices $A\in \calS(G)$.

It was mentioned in~\cite{Lin20SSPgraph} that it can be proved that graph $G_{104}=T_{5,1}$ is an SSP graph. In the following example, we provide a proof which illustrates the definition of a verification matrix, and the result will be needed in Section~\ref{sec:girth=5}.

\begin{example}\label{ex:t51ssp}
 We observe that the verification matrix of an arbitrary matrix of the form $$A'=\left(
\begin{array}{cccccc}
 a_{1,1} & a_{1,2} & 0 & 0 & a_{1,5} & 0 \\
 a_{1,2} & a_{2,2} & a_{2,3} & 0 & 0 & 0 \\
 0 & a_{2,3} & a_{3,3} & a_{3,4} & 0 & 0 \\
 0 & 0 & a_{3,4} & a_{4,4} & a_{4,5} & 0 \\
 a_{1,5} & 0 & 0 & a_{4,5} & a_{5,5} & a_{5,6} \\
 0 & 0 & 0 & 0 & a_{5,6} & a_{6,6} \\
\end{array}
\right) \in \calS(T_{5,1})$$
is equal to
\newcommand\scalemath[2]{\scalebox{#1}{\mbox{\ensuremath{\displaystyle #2}}}}
\[
 \Psi_S(A')=\left(
\scalemath{0.7}{
\begin{array}{ccccccccc}
 -a_{3,4} & a_{1,1}-a_{4,4} & 0 & a_{1,2} & 0 & 0 & 0 & 0 & 0 \\
 a_{1,2} & 0 & 0 & -a_{3,4} & 0 & 0 & 0 & 0 & 0 \\
 0 & -a_{4,5} & -a_{5,6} & 0 & a_{1,2} & 0 & 0 & 0 & 0 \\
 -a_{2,3} & 0 & 0 & 0 & a_{1,5} & 0 & 0 & 0 & 0 \\
 0 & a_{1,2} & 0 & a_{2,2}-a_{4,4} & -a_{4,5} & 0 & 0 & 0 & 0 \\
 0 & 0 & a_{1,1}-a_{6,6} & 0 & 0 & a_{1,2} & 0 & 0 & 0 \\
 a_{1,1}-a_{3,3} & -a_{3,4} & 0 & 0 & 0 & 0 & a_{1,5} & 0 & 0 \\
 0 & 0 & 0 & -a_{4,5} & a_{2,2}-a_{5,5} & -a_{5,6} & a_{2,3} & 0 & 0 \\
 0 & 0 & 0 & a_{2,3} & 0 & 0 & -a_{4,5} & 0 & 0 \\
 0 & 0 & a_{1,2} & 0 & -a_{5,6} & a_{2,2}-a_{6,6} & 0 & a_{2,3} & 0 \\
 -a_{1,5} & 0 & 0 & 0 & a_{2,3} & 0 & a_{3,3}-a_{5,5} & -a_{5,6} & 0 \\
 0 & 0 & 0 & 0 & 0 & a_{2,3} & -a_{5,6} & a_{3,3}-a_{6,6} & a_{3,4} \\
 0 & 0 & a_{1,5} & 0 & 0 & 0 & 0 & 0 & a_{4,5} \\
 0 & -a_{1,5} & 0 & 0 & 0 & 0 & a_{3,4} & 0 & -a_{5,6} \\
 0 & 0 & 0 & 0 & 0 & 0 & 0 & a_{3,4} & a_{4,4}-a_{6,6} \\
\end{array}}
\right).
\]
If we reorder columns and rows of $\Psi_S(A')$, we can obtain matrix
\[
 \Psi'=\left(
\scalemath{0.7}{
\begin{array}{ccccccccc}
  -a_{3,4} & a_{1,1}-a_{4,4} & a_{1,2} & 0 & 0 & 0 & 0 & 0 & 0 \\
 a_{1,2} & 0 & -a_{3,4} & 0 & 0 & 0 & 0 & 0 & 0 \\
 
 -a_{2,3} & 0 & 0 & a_{1,5} & 0 & 0 & 0 & 0 & 0 \\
 0 & a_{1,2} & a_{2,2}-a_{4,4} & -a_{4,5} & 0 & 0 & 0 & 0 & 0 \\
 a_{1,1}-a_{3,3} & -a_{3,4} & 0 & 0 & a_{1,5} & 0 & 0 & 0 & 0 \\

 0 & 0 & a_{2,3} & 0 & -a_{4,5} & 0 & 0 & 0 & 0 \\0 & -a_{4,5} & 0 & a_{1,2} & 0 & -a_{5,6} & 0 & 0 & 0 \\
  0 & 0 & -a_{4,5} & a_{2,2}-a_{5,5} & a_{2,3} & 0 & -a_{5,6} & 0 & 0 \\
 -a_{1,5} & 0 & 0 & a_{2,3} & a_{3,3}-a_{5,5} & 0 & 0 & -a_{5,6} & 0 \\
 0 & -a_{1,5} & 0 & 0 & a_{3,4} & 0 & 0 & 0 & -a_{5,6} \\
 0 & 0 & 0 & 0 & 0 & a_{1,1}-a_{6,6} & a_{1,2} & 0 & 0 \\
 0 & 0 & 0 & -a_{5,6} & 0 & a_{1,2} & a_{2,2}-a_{6,6} & a_{2,3} & 0 \\
 0 & 0 & 0 & 0 & -a_{5,6} & 0 & a_{2,3} & a_{3,3}-a_{6,6} & a_{3,4} \\
 0 & 0 & 0 & 0 & 0 & 0 & 0 & a_{3,4} & a_{4,4}-a_{6,6} \\
 0 & 0 & 0 & 0 & 0 & a_{1,5} & 0 & 0 & a_{4,5}
\end{array}}
\right),
\]
which is permutationally similar to $\Psi_S(A')$.
Notice that the submatrix $\Psi'_{\alpha}=\Psi_S(A')[\alpha,\colon]$ consisting of the rows with indices in $$\alpha=\{\{2,3\},\{1,5\},\{1,2\},\{3,4\},\{5,6\},\{1,6\},\{2,6\},\{4,6\},\{4,5\}\},$$
is equal to
\begin{equation}\label{eq:invertible-submatrix-of-verification-matrix-T51} 
\Psi'_{\alpha}=
\begin{pmatrix}
 a_{1,2} & 0 & -a_{3,4} & 0 & 0 & 0 & 0 & 0 & 0 \\
 0 & -a_{4,5} & 0 & a_{1,2} & 0 & -a_{5,6} & 0 & 0 & 0 \\
 -a_{2,3} & 0 & 0 & a_{1,5} & 0 & 0 & 0 & 0 & 0 \\
 0 & 0 & a_{2,3} & 0 & -a_{4,5} & 0 & 0 & 0 & 0 \\
 0 & -a_{1,5} & 0 & 0 & a_{3,4} & 0 & 0 & 0 & -a_{5,6} \\
 0 & 0 & 0 & 0 & 0 & a_{1,1}-a_{6,6} & a_{1,2} & 0 & 0 \\
 0 & 0 & 0 & -a_{5,6} & 0 & a_{1,2} & a_{2,2}-a_{6,6} & a_{2,3} & 0 \\
 0 & 0 & 0 & 0 & 0 & 0 & 0 & a_{3,4} & a_{4,4}-a_{6,6} \\
 0 & 0 & 0 & 0 & 0 & a_{1,5} & 0 & 0 & a_{4,5} 
\end{pmatrix},
\end{equation}
and is invertible since $\det(\Psi'_{\alpha})=2 a_{1,2} a_{1,5} a_{2,3} a_{3,4}^2 a_{4,5}^2 a_{5,6}^2\ne 0$.  This implies that $\Psi'$ is of full column rank for any matrix $A'\in \calS(T_{5,1})$ and hence $A'$ has the SSP. Therefore, $T_{5,1}$ is an SSP graph.
\end{example}

In~\cite{Lin20SSPgraph} the authors introduced three rules that force some entries of the matrix $X\in \Sc(G^c)$ with $[A,X]=O$ to be equal to $0$. Those rules can be used to prove that $G\in \gssp$. In Lemma~\ref{lem:rule4} we prove a new rule. First, we recall one of the rules from~\cite{Lin20SSPgraph}.

\begin{lemma}\cite[Lemma~3.1]{Lin20SSPgraph}\label{lem:rule1}
Let $G$ be a graph and $G_{\ell}$ a supergraph of $G$ of the same order.  Suppose $A\in\calS(G)$, 
$X\in\Sc(G_{\ell}^c)$, and $[A,X]=O$.  If for some  $i,j,k \in V(G)$ the conditions 
\[ N_{G}[i]\cap N_{G_{\ell}}[j]^c=\{k\} \text{ and } N_G[j] \cap N_{G_{\ell}}[i]^c=\emptyset\]
hold, then $X\in\Sc(G_{\ell+1}^c)$ with $G_{\ell+1} = G_{\ell} + \{j,k\}$. 
\end{lemma}

The following two lemmas ensure several entries of the matrix $X\in \Sc(G^c)$ with $[A,X]=O$ to be equal to $0$ in one step only.

\begin{lemma}\cite[Corollary~3.8]{Lin20SSPgraph}\label{lem:induced-paths}
Let $G$ be a graph and $G_{\ell}$ a supergraph of $G$ with $V(G)=V(G_{\ell})$. Suppose a path $P$ is an induced subgraph of $G$ with vertices labelled in the path order as $p_1,\ldots,p_n$, and $E(P^d)\subseteq E(G_{\ell})$ for some $d$. Moreover, let 
\[ N_{G}[p_i]\cap N_{G_{\ell}}[p_{i+d}]^c\subseteq V(P) \text{ and } N_G[p_{i+d}] \cap N_{G_{\ell}}[p_i]^c\subseteq V(P)\]
for $i=1,\ldots,n-d-1$. 
If $A\in\calS(G)$, $X\in\Sc(G_{\ell}^c)$ and $[A,X]=O$, then $X\in\Sc(G_{\ell_1}^c)$, where 
$G_{\ell_1}=G_{\ell}+ \{\{p_i,p_{i+d+1}\}: 1 \leq i \leq n-d-1\}$,
i.e., $E(P^{d+1})\subseteq E(G_{\ell_1})$. 
\end{lemma}

\begin{lemma}\cite[Corollary~3.9]{Lin20SSPgraph}\label{lem:corofcor}
 Let a graph $G$ have an induced path $P$ on vertices (in the path order) $v_1, \ldots, v_m$ such that $v_m$ is the only vertex with edges to the rest of the graph. If $A\in\calS(G)$, $X\in\Sc(G^c)$, and $[A,X]=O$, then $X\in\Sc(G_{1}^c)$ with $G_{1} = G+ \{\{v_i,v_j\}\colon 1 \leq i < j \leq m\}$.
\end{lemma}

Note that these lemmas force some of the entries of $X$ to be equal to $0$. If we set them to $0$, then the corresponding verification matrix will have fewer columns (since there are fewer variables $x_{i,j}$ in our linear system) and fewer rows (since some of the rows will be equal to $\bf 0$). See Sections~\ref{sec:girth-4} and~\ref{sec:girth=5}, where we use smaller verification matrices via this interpretation.

\smallskip

In~\cite[Definition~2.2]{Lin20SSPgraph} the authors define a barbell partition of a graph $G$ and show that its existence guarantees $G\notin \gssp$. Using this concept, the following was proved.

\begin{lemma}\cite[Corollary~2.5]{Lin20SSPgraph}\label{lem:unicyclic-not-in-gssp}
If a unicyclic graph $G$ has a vertex $v$ such that $\deg(v)\geq 4$, or a vertex $u$ not contained in the cycle with $\deg(u)\geq 3$, then $G\notin\gssp$.
\end{lemma}

Hence if a unicyclic graph is in $\gssp$, it must be obtained by joining a cycle $C_{m}$ to $k$ paths $P_{n_1},\ldots,P_{n_k}$,  $0\leq k\leq m$, $n_i\in \bN$, by bridges from distinct vertices of the cycle. It is known that not all such graphs are in $\gssp$, see e.g.~\cite[Example 5.10]{Lin20SSPgraph}.

\begin{lemma}\label{lem:known-ssp-graphs}
The following families of unicyclic graphs are in $\gssp$:
\begin{enumerate}
   \item Tadpole graphs $T_{3,n}\in \gssp$ for $n\in\bN_0$, see~\cite[Example~4.3]{Lin20SSPgraph} for $m=3$, and $T_{4,1}$, see~\cite[Example~3.17]{Lin20SSPgraph}.
   \item\label{item:k-at-most-2} Unicyclic graph of girth $3$ obtained by joining a cycle $C_{3}$ to two paths by two bridges from distinct vertices of the cycle, see~\cite[Theorem~51]{Barrett17GeneralizationsSAP}  and~\cite[Theorem~5.2]{Lin20SSPgraph}.
\end{enumerate}
\end{lemma}

\section{Unicyclic graphs of girth three}\label{sec:girth=3}

In this Section, we extend the characterization of trees in $\gssp$, \cite{Lin20SSPgraph}, to the complete characterization of unicyclic graphs of girth three in $\gssp$. We will prove that for unicyclic graphs of girth three, the necessary condition from Lemma~\ref{lem:unicyclic-not-in-gssp} is also sufficient.

If $G$ is a unicyclic graph in $\gssp$ of girth three, it must be obtained by joining the cycle $C_{3}$ to $k$ paths $P_{n_1},\ldots,P_{n_k}$, $0\leq k \leq 3$, $n_i\in \bN$, by bridges from distinct vertices of the cycle $C_3$. 
In the case $k \leq 2$, such graphs are in $\gssp$ by Lemma~\ref{lem:known-ssp-graphs}. 
Therefore, we only have to resolve the case $k=3$. We will show in Theorem~\ref{thm:unicyclic-girth-3-graphs} that all such graphs are in $\gssp$.

\newcommand{\exscale}{0.6}
 \begin{center}
 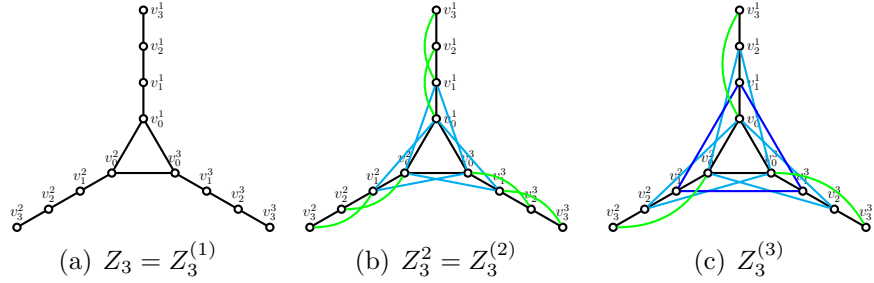
\begin{figure}[htb!]
 \centering
 \subfigure	[$Z_3=Z_3^{(1)}$]{
    \begin{tikzpicture}[scale=\exscale,transform shape]

\node[label={right:$v_0^1$}] (v01)   at (90:1) {}; 
\node[label={above:$v_0^2$}] (v02)   at (210:1) {}; 
\node[label={above:$v_0^3$}] (v03)   at (330:1) {}; 

\draw (v01) -- (v02) -- (v03)--(v01);

\node[label={right:$v_1^1$}] (v11)   at (90:2) {}; 
\node[label={above:$v_1^2$}] (v12)   at (210:2) {}; 
\node[label={above:$v_1^3$}] (v13)   at (330:2) {}; 
\node[label={right:$v_2^1$}] (v21)   at (90:3) {}; 
\node[label={above:$v_2^2$}] (v22)   at (210:3) {}; 
\node[label={above:$v_2^3$}] (v23)   at (330:3) {}; 
\node[label={right:$v_3^1$}] (v31)   at (90:4) {}; 
\node[label={above:$v_3^2$}] (v32)   at (210:4) {}; 
\node[label={above:$v_3^3$}] (v33)   at (330:4) {}; 

\draw (v01) -- (v11) -- (v21) -- (v31);
\draw (v02) -- (v12) -- (v22) -- (v32);
\draw (v03) -- (v13) -- (v23) -- (v33);

    \end{tikzpicture}\label{Z3}}
  \subfigure	[$Z_3^{2}=Z_3^{(2)}$]{
    \begin{tikzpicture}[scale=\exscale,transform shape]

\node[label={right:$v_0^1$}] (v01)   at (90:1) {}; 
\node[label={above:$v_0^2$}] (v02)   at (210:1) {}; 
\node[label={above:$v_0^3$}] (v03)   at (330:1) {}; 

\draw (v01) -- (v02) -- (v03)--(v01);

\node[label={right:$v_1^1$}] (v11)   at (90:2) {}; 
\node[label={above:$v_1^2$}] (v12)   at (210:2) {}; 
\node[label={above:$v_1^3$}] (v13)   at (330:2) {}; 
\node[label={right:$v_2^1$}] (v21)   at (90:3) {}; 
\node[label={above:$v_2^2$}] (v22)   at (210:3) {}; 
\node[label={above:$v_2^3$}] (v23)   at (330:3) {}; 
\node[label={right:$v_3^1$}] (v31)   at (90:4) {}; 
\node[label={above:$v_3^2$}] (v32)   at (210:4) {}; 
\node[label={above:$v_3^3$}] (v33)   at (330:4) {}; 

\draw (v01) -- (v11) -- (v21) -- (v31);
\draw (v02) -- (v12) -- (v22) -- (v32);
\draw (v03) -- (v13) -- (v23) -- (v33);

\begin{scope}[cyan]
\draw (v01) -- (v12) -- (v03) -- (v11) -- (v02) -- (v13) -- (v01);
\end{scope}

\begin{scope}[green]
\draw (v01) to[bend left] (v21);
\draw (v11) to[bend left] (v31);
\draw (v02) to[bend left] (v22);
\draw (v12) to[bend left] (v32);
\draw (v03) to[bend left] (v23);
\draw (v13) to[bend left] (v33);
\end{scope}

    \end{tikzpicture}   \label{Z32}} 
   \subfigure	[$Z_3^{(3)}$]{
    \begin{tikzpicture}[scale=\exscale,transform shape]

\node[label={right:$v_0^1$}] (v01)   at (90:1) {}; 
\node[label={above:$v_0^2$}] (v02)   at (210:1) {}; 
\node[label={above:$v_0^3$}] (v03)   at (330:1) {}; 

\draw (v01) -- (v02) -- (v03)--(v01);

\node[label={right:$v_1^1$}] (v11)   at (90:2) {}; 
\node[label={above:$v_1^2$}] (v12)   at (210:2) {}; 
\node[label={above:$v_1^3$}] (v13)   at (330:2) {}; 
\node[label={right:$v_2^1$}] (v21)   at (90:3) {}; 
\node[label={above:$v_2^2$}] (v22)   at (210:3) {}; 
\node[label={above:$v_2^3$}] (v23)   at (330:3) {}; 
\node[label={right:$v_3^1$}] (v31)   at (90:4) {}; 
\node[label={above:$v_3^2$}] (v32)   at (210:4) {}; 
\node[label={above:$v_3^3$}] (v33)   at (330:4) {}; 

\draw (v01) -- (v11) -- (v21) -- (v31);
\draw (v02) -- (v12) -- (v22) -- (v32);
\draw (v03) -- (v13) -- (v23) -- (v33);

\begin{scope}[cyan]
\draw (v01) -- (v22) -- (v03) -- (v21) -- (v02) -- (v23) -- (v01);
\end{scope}

\begin{scope}[blue]
\draw (v11) -- (v12) -- (v13) -- (v11);
\end{scope}

\begin{scope}[green]
\draw (v01) to[bend left] (v31);
\draw (v02) to[bend left] (v32);
\draw (v03) to[bend left] (v33);
\end{scope}

    \end{tikzpicture}   \label{Z33}} 
    \caption{Graph $Z_3$ together with its strong powers $Z_3^{(r)}$, $r\leq 3$.}\label{fig:Z3-strong-power}
\end{figure}
\end{center}

 Let $Z_h$ be the graph on $3h+3$ vertices obtained by joining the cycle $C_{3}$ to three paths $P_{h}$ by three bridges from distinct vertices of the cycle $C_3$, see Figure~\ref{fig:Z3-strong-power}.  
First we prove a similar lemma as in~\cite[Lemma~3.13]{Lin20SSPgraph}; 
if $G$ is a particular graph,  $A\in \calS(G)$, and $X\in \Sc(G^c)$, such that $[A,X]=O$ and the entries in $X$ in the positions of $Z_h^{h}$ are equal to $0$,  then the following lemma will guarantee that entries in $X$ in the positions of $Z_h^{(h+1)}$ are also equal to $0$.

\begin{lemma}\label{lem:rule4}
Let $G$ be a graph and $G_{\ell}$ a supergraph of $G$ of the same order.  Suppose $A\in\calS(G)$, $X\in \Sc(G_l^c)$, and $[A,X]=O$.  If $Z_h$ is an induced subgraph of $G$ such that  
\begin{enumerate}
\item $E(Z_h^{h})\subseteq E(G_{\ell})$ and $E(Z_h^{(h+1)})\cap E(G_{\ell})=\emptyset$, and 
\item for  all $\{u,v\}\in E(Z_h^{(h)})$ we have $N_G[u]\cap N_{G_{\ell}}[v]^c \subseteq V(Z_h)$ and $N_G[v] \cap N_{G_{\ell}}[u]^c\subseteq V(Z_h)$,
\end{enumerate}
then $X\in \Sc(G_{l+1}^c)$ with $G_{\ell+1}=G_{\ell} +E(Z_h^{(h+1)})$.
\end{lemma}
\begin{proof}
 Let us name the vertices of $Z_h$, $E(Z_h^{h})\subseteq E(G_{\ell})$, as the following:  let $v_0^1, v_0^2, v_0^3$ be the vertices on the cycle $C_3$ of $Z_h$ and let $P_{(i)}$ be a path on $h$ vertices $v^i_1,\ldots ,v^i_h$ for $i=1,2,3$. Suppose $Z_h$ is constructed from $P_{(1)}\cup P_{(2)}\cup P_{(3)}$ by joining vertices $v_0^i$ and $v_1^i$ by a bridge, $i=1,2,3$.
Let $ I=\{i \in \mathbb{N}_0 :  0\leq i\leq \left\lfloor{\frac{h-1}{2}} \right\rfloor\}$ and $J=\{j \in \mathbb{N}_0 :  0\leq j\leq \left\lfloor{\frac{h}{2}} \right\rfloor\}$. 
Let $$\alpha'=\{\{v_0^1, v_h^1\},\{v_0^2, v_h^2\}, \{v_0^3, v_h^3\}\}$$ and 
$$\alpha_i=\begin{cases}
    \{\{v_i^p,v_{h-1-i}^q\}\colon p,q\in \{1,2,3\}, p \ne q\},& \text{if } i < \frac{h-1}{2},\\
    \{\{v_i^p,v_{i}^q\}\colon p,q\in \{1,2,3\}, p \ne q\},& \text{if } i = \frac{h-1}{2},\\
\end{cases}$$
and let $\alpha= \bigcup \limits_{ i \in I} \alpha_i \cup \alpha'$. In the proof, $\alpha$ will be the set of positions, for which we will consider equation $[A,X]\vert_\alpha=0$.
Moreover, for $p,q\in \{1,2, 3\}$, $p\ne q$, $j\in J$, let us denote the edges $
e^{p,q}_{j}=\{v_j^p,v_{h-j}^q\}$ 
and observe that $e^{p,q}_j=e^{q,p}_{h-j}$. Now, let $$\beta_j=\{e_j^{1,2},e_j^{2,1},e_j^{1,3},e_j^{3,1},e_j^{2,3},e_j^{3,2}\}$$
and note that in particular, when $j=\frac{h}{2}$, we have $\beta_{h/2}=\{e_{h/2}^{1,2},e_{h/2}^{1,3},e_{h/2}^{2,3}\}.$ 
Let $\beta=\bigcup \limits_ {j \in J} \beta_j$ and note that $\beta$ is the set of edges of $E(Z_h^{(h+1)}).$
Observe that by condition (2) of the lemma, all equations coming from $\alpha$ and including variables in $\beta$ are binomial, and so the matrix $M$ corresponding to the system $M\bfx={\bf 0}$, is permutationally equivalent to
\[
\begin{bmatrix}
a_{v_0^1v^3_0} & & & & & & & &a_{v_0^1v^2_0} \\
-a_{v^1_{h-1}v^1_h} &a_{v^3_0v^3_1}  & & & & & & & \\
 & \ddots & \ddots &  & & & & & \\
 &  &a_{v_0^1v^3_0} & a_{v_0^2v^3_0} & & & & & \\ 
 & & & -a_{v_{h-1}^3v^3_h} & a_{v^2_0v^2_1} & & & & \\
 & & & & \ddots & \ddots & & & \\ 
 & & & & &  a_{v_0^2v^3_0}& a_{v_0^1v^2_0} & & \\ 
 & & & & & &  -a_{v^2_{h-1}v^2_{h}}& a_{v^1_{0}v^1_1} & \\ 
 & & & & & & & \ddots & \ddots \\ 
\end{bmatrix}.
\]
Here, the variables $\bfx
$, arise from the edges of $\beta$. Since $|\det(M)|$ is twice the product of its diagonal entries, it follows that $\det(M)\ne 0$ 
for any choice of $A \in S(G)$. Therefore $\bfx=0$ and $X\in\Sc(G_{\ell+1}^c)$ with $G_{\ell+1}=G_{\ell} +E(Z_h^{(h+1)})$.
\end{proof}

Now, we prove the following theorem, which characterizes unicyclic graphs of girth three in $\gssp$. The idea of the proof follows the concept of the proof for trees in $\gssp$,~\cite[Theorem~4.3]{Lin20SSPgraph} and the result extends the characterization of trees in $\gssp$ to unicyclic graphs of girth $3$. 

\begin{theorem}\label{thm:unicyclic-girth-3-graphs}
    A unicyclic graph of girth three is in $\gssp$ if and only if it is obtained by joining the cycle $C_{3}$ to $k$ paths $P_{n_1},\ldots,P_{n_k}$,  $0\leq k \leq 3$, $n_i\in \bN$, by bridges from distinct vertices of the cycle $C_3$.
\end{theorem}
\begin{proof}
   As written in the first paragraph of this Section, we only need to prove that if $G$ is a graph obtained by joining the cycle $C_{3}$ to three paths $P_{n_1},P_{n_2},P_{n_3}$, $n_i\in \bN$, by bridges from distinct vertices of the cycle $C_3$, then $G\in \gssp$. Without loss of generality, we assume $n_1\leq n_2 \leq n_3.$ Notice that $Z_h$ is a subgraph of $G$, for all $h\in\{1,2,...,n_1\}$, but $Z_{n_1+1}$ is not a subgraph of $G$. 
   
   Suppose $A \in \calS(G)$ is an arbitrary matrix and $X=[x_{i,j}]$ is symmetric matrix such that $[A,X]=O$ and $A \circ X=I \circ X=O.$
   First we apply Lemma~\ref{lem:corofcor} three times, once for each path $P_{n_1}, P_{n_2}, P_{n_3}$, taking $v_1$ to be the leaf of $P_{n_i}, i\in\{1,2,3\}$. We conclude that $X\in \Sc(G_1^c)$ where $G_1$ is obtained from $G$ by adding all possible edges in each of the paths $P_{n_1}, P_{n_2}, P_{n_3}$. For a concrete example when $(n_1, n_2,n_3)=(2,3,4)$, see Figure~\ref{ex:G1}.

   Since $Z_h$ is an induced subgraph of $G$, for all $h=1,...,n_1$, we will inductively use Lemma~\ref{lem:rule4}  to conclude that in each step $X \in \Sc(G_{h+1}^c)$, where $G_{h+1}$ is obtained from $G_h$ by adding all possible edges between the vertices which are on distance exactly $h+1$. Note that in the case $h=1$ we have $Z_1=Z_1^1=Z_1^{(1)}$ and hence part (1) of Lemma~\ref{lem:rule4} is fulfilled. Moreover, if $\{u,v\}\in E(Z_1)$, then $N_G[u]\cap N_{G_{1}}[v]^c \subseteq V(Z_1)$ and $N_G[v]\cap N_{G_{1}}[u]^c \subseteq V(Z_1)$. Namely, if $u$ and $v$ are both vertices of the cycle, then these conditions are trivially fulfilled. If not both of them are on the cycle and since $\{u,v\}\in E(Z_1)$, assume without loss of generality that $u$ is on the cycle and $v$ is not. Note that both vertices $u$ and $v$ are the neighbours of all the vertices on the path they are on, so $N_G[u]\cap N_{G_{1}}[v]^c \subseteq V(Z_1)$ and $N_G[v]\cap N_{G_{1}}[u]^c=\emptyset$.  Thus part (2) of Lemma~\ref{lem:rule4} is true, and we have $X\in \Sc(G_{2})$, where $G_2=G_1+E(Z_1^{(2)})$.
    Suppose now $2 \leq h\leq n_1$ is arbitrary. Notice that conditions $E(Z_h^h)\subseteq E(G_h)$ and $E(Z_h^{(h+1)})\cap E(G_h)=\emptyset$ are satisfied because we inductively add edges between vertices at distance at most $h+1$ to $G_h$ in each step. Let us now check the second condition of Lemma~\ref{lem:rule4}. Suppose $\{u,v\} \in E(Z_h^{(h)})$. If neither $u$ nor $v$ is a vertex on the cycle, then since they are on distance $h$, they are not the leaves on the paths $P_{n_1}, P_{n_2}, P_{n_3}$  of graph $G$ nor the leaves of $Z_h$. This means $N_G[u], N_G[v] \subseteq V(Z_h)$ and so part (2) is true. Suppose now $u\in V(C_3)$, which implies $N_G[u]\subseteq V(Z_h)$. If the shortest path from $u$ to $v$ contains another vertex of $C_3$, then $N_G[v]\subseteq V(Z_h)$. Otherwise, there exists $z\in N_G[v]\setminus V(Z_h)$ at distance $h+1$ from $u$, but $\{u,z\}\in E(G_1)\subseteq E(G_h)$. Hence $N_G[v]\cap N_{G_h}[u]^c=\emptyset$. In either case (2) from Lemma~\ref{lem:rule4} is fulfilled and hence $x_{i,j}=0$ for all $i, j \in V(G)$ which are in $G$ at distance at most $h+1$, i.e. for all $\{i,j\}\in E(G_{h+1})$. At the end of this induction, we have that $X \in \Sc(G_{n_1+1}^c)$, so $G^{n_1+1}\subseteq G_{n_1+1}$. 
   For an example of two steps in the induction see Figures~\ref{ex:G2} and~\ref{ex:G3} when $(n_1, n_2,n_3)=(2,3,4)$.
 
 \begin{center}
 \begin{figure}[htb]
 \centering
 \subfigure	[$G_0=G$]{
    \begin{tikzpicture}[scale=\exscale,transform shape]

\node[label={right:$v_0^1$}] (v01)   at (90:1) {}; 
\node[label={above:$v_0^2$}] (v02)   at (210:1) {}; 
\node[label={above:$v_0^3$}] (v03)   at (330:1) {}; 

\draw (v01) -- (v02) -- (v03)--(v01);

\node[label={right:$v_1^1$}] (v11)   at (90:2) {}; 
\node[label={above:$v_1^2$}] (v12)   at (210:2) {}; 
\node[label={above:$v_1^3$}] (v13)   at (330:2) {}; 
\node[label={right:$v_2^1$}] (v21)   at (90:3) {}; 
\node[label={above:$v_2^2$}] (v22)   at (210:3) {}; 
\node[label={above:$v_2^3$}] (v23)   at (330:3) {}; 
\node[label={right:$v_3^1$}] (v31)   at (90:4) {}; 
\node[label={above:$v_3^2$}] (v32)   at (210:4) {}; 
\node[label={right:$v_4^1$}] (v41)   at (90:5) {}; 

\draw (v01) -- (v11) -- (v21) -- (v31) -- (v41);
\draw (v02) -- (v12) -- (v22) -- (v32);
\draw (v03) -- (v13) -- (v23);
    \end{tikzpicture}\label{ex:G}}
    \qquad
  \subfigure	[$G_1$]{
    \begin{tikzpicture}[scale=\exscale,transform shape]

\node[label={right:$v_0^1$}] (v01)   at (90:1) {}; 
\node[label={above:$v_0^2$}] (v02)   at (210:1) {}; 
\node[label={above:$v_0^3$}] (v03)   at (330:1) {}; 

\draw (v01) -- (v02) -- (v03)--(v01);

\node[label={right:$v_1^1$}] (v11)   at (90:2) {}; 
\node[label={above:$v_1^2$}] (v12)   at (210:2) {}; 
\node[label={above:$v_1^3$}] (v13)   at (330:2) {}; 
\node[label={right:$v_2^1$}] (v21)   at (90:3) {}; 
\node[label={above:$v_2^2$}] (v22)   at (210:3) {}; 
\node[label={above:$v_2^3$}] (v23)   at (330:3) {}; 
\node[label={right:$v_3^1$}] (v31)   at (90:4) {}; 
\node[label={above:$v_3^2$}] (v32)   at (210:4) {}; 
\node[label={right:$v_4^1$}] (v41)   at (90:5) {}; 

\draw (v01) -- (v11) -- (v21) -- (v31) -- (v41);
\draw (v02) -- (v12) -- (v22) -- (v32);
\draw (v03) -- (v13) -- (v23);

\begin{scope}[blue]
\draw (v01) to[bend left] (v21) to[bend left] (v41);
\draw (v11) to[bend left] (v31);
\draw (v01) to[bend left] (v31);
\draw (v11) to[bend left] (v41);
\draw (v01) to[bend left] (v41);
\draw (v02) to[bend left] (v22);
\draw (v12) to[bend left] (v32);
\draw (v02) to[bend left] (v32);
\draw (v03) to[bend left] (v23);
\end{scope}
    \end{tikzpicture}   \label{ex:G1}} 
    \qquad
   \subfigure	[$G_2$]{
    \begin{tikzpicture}[scale=\exscale,transform shape]

\node[label={right:$v_0^1$}] (v01)   at (90:1) {}; 
\node[label={above:$v_0^2$}] (v02)   at (210:1) {}; 
\node[label={above:$v_0^3$}] (v03)   at (330:1) {}; 

\draw (v01) -- (v02) -- (v03)--(v01);

\node[label={right:$v_1^1$}] (v11)   at (90:2) {}; 
\node[label={above:$v_1^2$}] (v12)   at (210:2) {}; 
\node[label={above:$v_1^3$}] (v13)   at (330:2) {}; 
\node[label={right:$v_2^1$}] (v21)   at (90:3) {}; 
\node[label={above:$v_2^2$}] (v22)   at (210:3) {}; 
\node[label={above:$v_2^3$}] (v23)   at (330:3) {}; 
\node[label={right:$v_3^1$}] (v31)   at (90:4) {}; 
\node[label={above:$v_3^2$}] (v32)   at (210:4) {}; 
\node[label={right:$v_4^1$}] (v41)   at (90:5) {}; 

\draw (v01) -- (v11) -- (v21) -- (v31) -- (v41);
\draw (v02) -- (v12) -- (v22) -- (v32);
\draw (v03) -- (v13) -- (v23);

\begin{scope}[gray]
\draw (v01) to[bend left] (v21) to[bend left] (v41);
\draw (v11) to[bend left] (v31);
\draw (v01) to[bend left] (v31);
\draw (v11) to[bend left] (v41);
\draw (v01) to[bend left] (v41);
\draw (v02) to[bend left] (v22);
\draw (v12) to[bend left] (v32);
\draw (v02) to[bend left] (v32);
\draw (v03) to[bend left] (v23);
\end{scope}

\begin{scope}[blue]
\draw (v01) -- (v12) -- (v03) -- (v11) -- (v02) -- (v13) -- (v01);
\end{scope}
   \end{tikzpicture}   \label{ex:G2}} 
   \qquad
   \subfigure	[$G_3$]{
    \begin{tikzpicture}[scale=\exscale,transform shape]

\node[label={right:$v_0^1$}] (v01)   at (90:1) {}; 
\node[label={above:$v_0^2$}] (v02)   at (210:1) {}; 
\node[label={above:$v_0^3$}] (v03)   at (330:1) {}; 

\draw (v01) -- (v02) -- (v03)--(v01);

\node[label={right:$v_1^1$}] (v11)   at (90:2) {}; 
\node[label={above:$v_1^2$}] (v12)   at (210:2) {}; 
\node[label={above:$v_1^3$}] (v13)   at (330:2) {}; 
\node[label={right:$v_2^1$}] (v21)   at (90:3) {}; 
\node[label={above:$v_2^2$}] (v22)   at (210:3) {}; 
\node[label={above:$v_2^3$}] (v23)   at (330:3) {}; 
\node[label={right:$v_3^1$}] (v31)   at (90:4) {}; 
\node[label={above:$v_3^2$}] (v32)   at (210:4) {}; 
\node[label={right:$v_4^1$}] (v41)   at (90:5) {}; 

\draw (v01) -- (v11) -- (v21) -- (v31) -- (v41);
\draw (v02) -- (v12) -- (v22) -- (v32);
\draw (v03) -- (v13) -- (v23);

\begin{scope}[gray]
\draw (v01) to[bend left] (v21) to[bend left] (v41);
\draw (v11) to[bend left] (v31);
\draw (v01) to[bend left] (v31);
\draw (v11) to[bend left] (v41);
\draw (v01) to[bend left] (v41);
\draw (v02) to[bend left] (v22);
\draw (v12) to[bend left] (v32);
\draw (v02) to[bend left] (v32);
\draw (v03) to[bend left] (v23);
\draw (v01) -- (v12) -- (v03) -- (v11) -- (v02) -- (v13) -- (v01);
\end{scope}

\begin{scope}[blue]
\draw (v01) -- (v22) -- (v03) -- (v21) -- (v02) -- (v23) -- (v01);
\draw (v11) -- (v12) -- (v13) -- (v11);
\end{scope}

    \end{tikzpicture}   \label{ex:G3}} 
    \qquad
       \subfigure	[$G_4$]{
    \begin{tikzpicture}[scale=\exscale,transform shape]

\node[label={right:$v_0^1$}] (v01)   at (90:1) {}; 
\node[label={above:$v_0^2$}] (v02)   at (210:1) {}; 
\node[label={above:$v_0^3$}] (v03)   at (330:1) {}; 

\draw (v01) -- (v02) -- (v03)--(v01);

\node[label={right:$v_1^1$}] (v11)   at (90:2) {}; 
\node[label={above:$v_1^2$}] (v12)   at (210:2) {}; 
\node[label={above:$v_1^3$}] (v13)   at (330:2) {}; 
\node[label={right:$v_2^1$}] (v21)   at (90:3) {}; 
\node[label={above:$v_2^2$}] (v22)   at (210:3) {}; 
\node[label={above:$v_2^3$}] (v23)   at (330:3) {}; 
\node[label={right:$v_3^1$}] (v31)   at (90:4) {}; 
\node[label={above:$v_3^2$}] (v32)   at (210:4) {}; 
\node[label={right:$v_4^1$}] (v41)   at (90:5) {}; 

\draw (v01) -- (v11) -- (v21) -- (v31) -- (v41);
\draw (v02) -- (v12) -- (v22) -- (v32);
\draw (v03) -- (v13) -- (v23);

\begin{scope}[gray]
\draw (v01) to[bend left] (v21) to[bend left] (v41);
\draw (v11) to[bend left] (v31);
\draw (v01) to[bend left] (v31);
\draw (v11) to[bend left] (v41);
\draw (v01) to[bend left] (v41);
\draw (v02) to[bend left] (v22);
\draw (v12) to[bend left] (v32);
\draw (v02) to[bend left] (v32);
\draw (v03) to[bend left] (v23);
\draw (v01) -- (v12) -- (v03) -- (v11) -- (v02) -- (v13) -- (v01);

\draw (v01) -- (v22) -- (v03) -- (v21) -- (v02) -- (v23) -- (v01);
\draw (v11) -- (v12) -- (v13) -- (v11);
\end{scope}

\begin{scope}[blue]
\draw (v12) -- (v21) -- (v13) -- (v22) -- (v11) -- (v23) -- (v12);
\draw (v01) -- (v32) -- (v03) -- (v31) -- (v02);
\end{scope}
    \end{tikzpicture}   \label{ex:G4}} 
    \qquad
      \subfigure	[$G_5$]{
    \begin{tikzpicture}[scale=\exscale,transform shape]

\node[label={right:$v_0^1$}] (v01)   at (90:1) {}; 
\node[label={above:$v_0^2$}] (v02)   at (210:1) {}; 
\node[label={above:$v_0^3$}] (v03)   at (330:1) {}; 

\draw (v01) -- (v02) -- (v03)--(v01);

\node[label={right:$v_1^1$}] (v11)   at (90:2) {}; 
\node[label={above:$v_1^2$}] (v12)   at (210:2) {}; 
\node[label={above:$v_1^3$}] (v13)   at (330:2) {}; 
\node[label={right:$v_2^1$}] (v21)   at (90:3) {}; 
\node[label={above:$v_2^2$}] (v22)   at (210:3) {}; 
\node[label={above:$v_2^3$}] (v23)   at (330:3) {}; 
\node[label={right:$v_3^1$}] (v31)   at (90:4) {}; 
\node[label={above:$v_3^2$}] (v32)   at (210:4) {}; 
\node[label={right:$v_4^1$}] (v41)   at (90:5) {}; 

\draw (v01) -- (v11) -- (v21) -- (v31) -- (v41);
\draw (v02) -- (v12) -- (v22) -- (v32);
\draw (v03) -- (v13) -- (v23);

\begin{scope}[gray]
\draw (v01) to[bend left] (v21) to[bend left] (v41);
\draw (v11) to[bend left] (v31);
\draw (v01) to[bend left] (v31);
\draw (v11) to[bend left] (v41);
\draw (v01) to[bend left] (v41);
\draw (v02) to[bend left] (v22);
\draw (v12) to[bend left] (v32);
\draw (v02) to[bend left] (v32);
\draw (v03) to[bend left] (v23);
\draw (v01) -- (v12) -- (v03) -- (v11) -- (v02) -- (v13) -- (v01);

\draw (v01) -- (v22) -- (v03) -- (v21) -- (v02) -- (v23) -- (v01);
\draw (v11) -- (v12) -- (v13) -- (v11);

\draw (v12) -- (v21) -- (v13) -- (v22) -- (v11) -- (v23) -- (v12);
\draw (v01) -- (v32) -- (v03) -- (v31) -- (v02);
\end{scope}
\begin{scope}[blue]
\draw (v03) -- (v41) -- (v02);
\draw (v32) -- (v13) -- (v31) -- (v12);
\draw (v22) -- (v23) -- (v21) -- (v22);
\draw          (v11) -- (v32);
\end{scope}
    \end{tikzpicture}   \label{ex:G5}}
        \qquad
      \subfigure	[$G_6$]{
    \begin{tikzpicture}[scale=\exscale,transform shape]
\node[label={right:$v_0^1$}] (v01)   at (90:1) {}; 
\node[label={above:$v_0^2$}] (v02)   at (210:1) {}; 
\node[label={above:$v_0^3$}] (v03)   at (330:1) {}; 

\draw (v01) -- (v02) -- (v03)--(v01);

\node[label={right:$v_1^1$}] (v11)   at (90:2) {}; 
\node[label={above:$v_1^2$}] (v12)   at (210:2) {}; 
\node[label={above:$v_1^3$}] (v13)   at (330:2) {}; 
\node[label={right:$v_2^1$}] (v21)   at (90:3) {}; 
\node[label={above:$v_2^2$}] (v22)   at (210:3) {}; 
\node[label={above:$v_2^3$}] (v23)   at (330:3) {}; 
\node[label={right:$v_3^1$}] (v31)   at (90:4) {}; 
\node[label={above:$v_3^2$}] (v32)   at (210:4) {}; 
\node[label={right:$v_4^1$}] (v41)   at (90:5) {}; 

\draw (v01) -- (v11) -- (v21) -- (v31) -- (v41);
\draw (v02) -- (v12) -- (v22) -- (v32);
\draw (v03) -- (v13) -- (v23);

\begin{scope}[gray]
\draw (v01) to[bend left] (v21) to[bend left] (v41);
\draw (v11) to[bend left] (v31);
\draw (v01) to[bend left] (v31);
\draw (v11) to[bend left] (v41);
\draw (v01) to[bend left] (v41);
\draw (v02) to[bend left] (v22);
\draw (v12) to[bend left] (v32);
\draw (v02) to[bend left] (v32);
\draw (v03) to[bend left] (v23);
\draw (v01) -- (v12) -- (v03) -- (v11) -- (v02) -- (v13) -- (v01);

\draw (v01) -- (v22) -- (v03) -- (v21) -- (v02) -- (v23) -- (v01);
\draw (v11) -- (v12) -- (v13) -- (v11);

\draw (v12) -- (v21) -- (v13) -- (v22) -- (v11) -- (v23) -- (v12);
\draw (v01) -- (v32) -- (v03) -- (v31) -- (v02);

\draw (v03) -- (v41) -- (v02);
\draw (v32) -- (v13) -- (v31) -- (v12);
\draw (v22) -- (v23) -- (v21) -- (v22);
\draw          (v11) -- (v32);
\end{scope}

\begin{scope}[blue]
\draw (v21) -- (v32) -- (v23) -- (v31) -- (v22);
\draw (v13) -- (v41) -- (v12);
\end{scope}
    \end{tikzpicture}   \label{ex:G6}}
    \qquad
        \subfigure	[$G_7$]{
    \begin{tikzpicture}[scale=\exscale,transform shape]

\node[label={right:$v_0^1$}] (v01)   at (90:1) {}; 
\node[label={above:$v_0^2$}] (v02)   at (210:1) {}; 
\node[label={above:$v_0^3$}] (v03)   at (330:1) {}; 

\draw (v01) -- (v02) -- (v03)--(v01);

\node[label={right:$v_1^1$}] (v11)   at (90:2) {}; 
\node[label={above:$v_1^2$}] (v12)   at (210:2) {}; 
\node[label={above:$v_1^3$}] (v13)   at (330:2) {}; 
\node[label={right:$v_2^1$}] (v21)   at (90:3) {}; 
\node[label={above:$v_2^2$}] (v22)   at (210:3) {}; 
\node[label={above:$v_2^3$}] (v23)   at (330:3) {}; 
\node[label={right:$v_3^1$}] (v31)   at (90:4) {}; 
\node[label={above:$v_3^2$}] (v32)   at (210:4) {}; 
\node[label={right:$v_4^1$}] (v41)   at (90:5) {}; 

\draw (v01) -- (v11) -- (v21) -- (v31) -- (v41);
\draw (v02) -- (v12) -- (v22) -- (v32);
\draw (v03) -- (v13) -- (v23);

\begin{scope}[gray]
\draw (v01) to[bend left] (v21) to[bend left] (v41);
\draw (v11) to[bend left] (v31);
\draw (v01) to[bend left] (v31);
\draw (v11) to[bend left] (v41);
\draw (v01) to[bend left] (v41);
\draw (v02) to[bend left] (v22);
\draw (v12) to[bend left] (v32);
\draw (v02) to[bend left] (v32);
\draw (v03) to[bend left] (v23);
\draw (v01) -- (v12) -- (v03) -- (v11) -- (v02) -- (v13) -- (v01);

\draw (v01) -- (v22) -- (v03) -- (v21) -- (v02) -- (v23) -- (v01);
\draw (v11) -- (v12) -- (v13) -- (v11);

\draw (v12) -- (v21) -- (v13) -- (v22) -- (v11) -- (v23) -- (v12);
\draw (v01) -- (v32) -- (v03) -- (v31) -- (v02);

\draw (v03) -- (v41) -- (v02);
\draw (v32) -- (v13) -- (v31) -- (v12);
\draw (v22) -- (v23) -- (v21) -- (v22);
\draw          (v11) -- (v32);

\draw (v21) -- (v32) -- (v23) -- (v31) -- (v22);
\draw (v13) -- (v41) -- (v12);
\end{scope}

\begin{scope}[blue]
\draw (v32) -- (v31);
\draw (v23) -- (v41) -- (v22);
\end{scope}
    \end{tikzpicture}   \label{ex:G7}}
    \qquad
     \subfigure	[$G_8=K_{12}$]{
    \begin{tikzpicture}[scale=\exscale,transform shape]

\node[label={right:$v_0^1$}] (v01)   at (90:1) {}; 
\node[label={above:$v_0^2$}] (v02)   at (210:1) {}; 
\node[label={above:$v_0^3$}] (v03)   at (330:1) {}; 

\draw (v01) -- (v02) -- (v03)--(v01);

\node[label={right:$v_1^1$}] (v11)   at (90:2) {}; 
\node[label={above:$v_1^2$}] (v12)   at (210:2) {}; 
\node[label={above:$v_1^3$}] (v13)   at (330:2) {}; 
\node[label={right:$v_2^1$}] (v21)   at (90:3) {}; 
\node[label={above:$v_2^2$}] (v22)   at (210:3) {}; 
\node[label={above:$v_2^3$}] (v23)   at (330:3) {}; 
\node[label={right:$v_3^1$}] (v31)   at (90:4) {}; 
\node[label={above:$v_3^2$}] (v32)   at (210:4) {}; 
\node[label={right:$v_4^1$}] (v41)   at (90:5) {}; 

\draw (v01) -- (v11) -- (v21) -- (v31) -- (v41);
\draw (v02) -- (v12) -- (v22) -- (v32);
\draw (v03) -- (v13) -- (v23);

\begin{scope}[gray]
\draw (v01) to[bend left] (v21) to[bend left] (v41);
\draw (v11) to[bend left] (v31);
\draw (v01) to[bend left] (v31);
\draw (v11) to[bend left] (v41);
\draw (v01) to[bend left] (v41);
\draw (v02) to[bend left] (v22);
\draw (v12) to[bend left] (v32);
\draw (v02) to[bend left] (v32);
\draw (v03) to[bend left] (v23);
\draw (v01) -- (v12) -- (v03) -- (v11) -- (v02) -- (v13) -- (v01);

\draw (v01) -- (v22) -- (v03) -- (v21) -- (v02) -- (v23) -- (v01);
\draw (v11) -- (v12) -- (v13) -- (v11);

\draw (v12) -- (v21) -- (v13) -- (v22) -- (v11) -- (v23) -- (v12);
\draw (v01) -- (v32) -- (v03) -- (v31) -- (v02);

\draw (v03) -- (v41) -- (v02);
\draw (v32) -- (v13) -- (v31) -- (v12);
\draw (v22) -- (v23) -- (v21) -- (v22);
\draw          (v11) -- (v32);

\draw (v21) -- (v32) -- (v23) -- (v31) -- (v22);
\draw (v13) -- (v41) -- (v12);
\draw (v32) -- (v31);
\draw (v23) -- (v41) -- (v22);
\end{scope}

\begin{scope}[blue]
\draw (v41) -- (v32);
\end{scope}
    \end{tikzpicture}   \label{ex:G8}}
    \caption{Graph $G$ and its supergraphs $G_{\ell}$, $\ell\in [8]$, illustrating the steps in the proof of Theorem~\ref{thm:unicyclic-girth-3-graphs}. The blue edges in $E(G_{\ell})$, $\ell\in [8]$, are the edges in $E(G_{\ell})\setminus E(G_{\ell-1})$, and the grey edges are the ones already in $ E(G_{\ell-1})$.}\label{fig:G}
\end{figure}
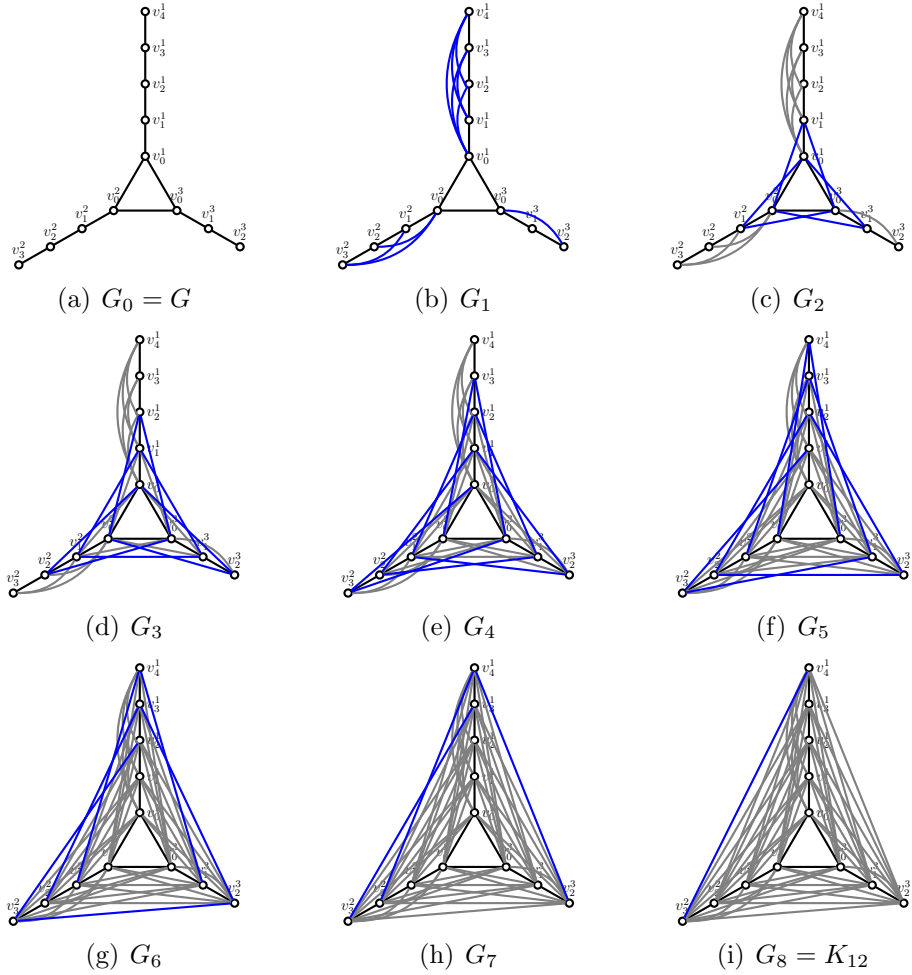
\end{center}
 Now, we repeatedly use Lemma~\ref{lem:induced-paths} to conclude that  $X=O$. First, we will add all pairs of vertices at the distance $n_1+2$ to $E(G_{n_1+1})$. Suppose that $v_i$ is the vertex of $P_{n_i}$, $i\in \{2,3\}$, such that its distance from the cycle $C_3$ is exactly $n_1+1.$ Let $Q_i$ be the path from vertex $v_i$ to the leaf of $P_{n_1}$. We apply Lemma~\ref{lem:induced-paths} to $Q_i$ with $d=n_1+1$, which implies that for all edges $\{u,v\}$, where $u$ and $v$ are at distance $n_1+2$ on that path, we have $x_{u,v}=0$. 
 Now, observe the path $Q$ from $v_2$ to $v_3$ and again apply Lemma~\ref{lem:induced-paths} to path $Q$ and $d=n_1+1$. It implies that for all edges $\{u,v\}$, where $u$ and $v$ are at distance $n_1+2$ on $Q$, we have $x_{u,v}=0$. 
Now, we define $G_{n_1+2}$ to be a supergraph of $G_{n_1+1}$ with all edges at distance $n_1+2$ added, so $G^{n_1+2}\subseteq G_{n_1+2}$. See Figure~\ref{ex:G4} as an example. Repeating this same process for all values $d$, $n_1+2\leq d\leq n_2+n_3+1$, we conclude that for $n=n_1+n_2+n_3+3, X \in \Sc(K_n^c)$, so $X=O$. For $d\in\{5,6,7,8\}$ in our working example see Figures~\ref{ex:G5}-\ref{ex:G8}.
\end{proof}

\section{Unicyclic graphs of girth four}\label{sec:girth-4}

The smallest graph of girth four, cycle $C_4$, is not an SSP graph as its adjacency matrix does not have the SSP. Hence, not all graphs described in Lemma~\ref{lem:unicyclic-not-in-gssp} are in $\gssp$. But on the other hand, the tadpole graph $T_{4,1}$ is an SSP graph,~\cite[Example 3.17]{Lin20SSPgraph}. In this Section, we prove that the tadpole graph $T_{4,n}$ is an SSP graph for all $n\geq 1$. We will provide several examples of unicyclic graphs of girth four that are not SSP graphs and thus showing that the necessary conditions for a girth four unicyclic graph to be in $\gssp$ from Lemma~\ref{lem:unicyclic-not-in-gssp} are not sufficient. 

 \subsection{Tadpole graphs $T_{4,n}$}\label{sec:girth-4-SSP}

In this subsection, we show that every tadpole graph of girth four is in $\gssp$.

\begin{theorem}\label{t4nssp}
Graph $T_{4,n}$  is an SSP graph, for any $n\in \mathbb{N}$.
\end{theorem}
\begin{proof}
Observe first that the statement for $n=1$ was proved in \cite[Example~3.17]{Lin20SSPgraph}. Suppose now $n\geq 2$ and let $$E(T_{4,n})=\{\{i,i+1\}\colon i \in [n+3]\} \; \cup \; \{\{1,4\}\}.$$
Now, let $A=[a_{i,j}]\in \calS(T_{4,n})$ and $X=[x_{i,j}]\in \Sc(T_{4,n}^c)$.
Observe that by Lemma~\ref{lem:corofcor} we have $x_{i,j}=0$ for all $4\leq i \leq j \leq n+4,$ so 
 $$X=
 \begin{pmatrix}
0&0&x_{1,3}&0&x_{1,5}&...&x_{1,n+4}\\
 0& 0 &0&x_{2,4}&x_{2,5}&...&x_{2,n+4}\\
 x_{1,3}&0 & 0 &0&x_{3,5}&...&x_{3,n+4}\\
0& x_{2,4}& 0& 0 & 0 &...&0\\
\vdots& \vdots& \vdots & \vdots& \vdots& & \vdots \\
x_{1,n+4}& x_{2,n+4}&x_{3,n+4} &0 & 0& \cdots & 0
\end{pmatrix}.$$
To show that all the entries of matrix $X$ are equal to zero, we first consider a subset of equations $[A, X]_{i,j}=0$ where $\{i, j\} \in \alpha_0=\{\{1,4\}, \{1, 2\}$ $\{2, 3\}, \{2, 4\}, \{4, 5\}\}$. Verify that the only variables in the five mentioned equations are $x_{1,3}$, $x_{1,5}$, $x_{2,4}$, $x_{2,5}$, $x_{3,5}$, so the corresponding matrix of the system of five equations on $\alpha_0$ is
$$\Psi_5=\left(
\begin{array}{ccccc}
 -a_{3,4} & -a_{4,5} & a_{1,2} & 0 & 0 \\
 -a_{2,3} & 0 & a_{1,4} & 0 & 0 \\
 a_{1,2} & 0 & -a_{3,4} & 0 & 0 \\
 0 & 0 & a_{2,2}-a_{4,4} & -a_{4,5} & 0 \\
 0 & a_{1,4} & 0 & 0 & a_{3,4} 
\end{array}
\right),$$
and it is a submatrix a verification matrix $\Psi_A$ of system $[A,X]=O$, and $\Psi_A[\alpha_0,\colon]=\begin{pmatrix}
    \Psi_5 & O
\end{pmatrix}$.
Note that $\det(\Psi_5)=(a_{1,2}a_{1,4}-a_{2,3}a_{3,4})a_{4,5}^2a_{3,4}$. We will consider two cases.

Assume first that $a_{1,2}a_{1,4}-a_{2,3}a_{3,4} \neq 0$, which implies $\det(\Psi_5)\ne 0$, and so $\Psi_5\bfx={\bf 0}$ has only trivial solution $x_{1,3}=x_{2,4}=x_{1,5}=x_{2,5}=x_{3,5}=0$.
In the case when $a_{1,2}a_{1,4}-a_{2,3}a_{3,4}=0$ and $n\geq 3$, let $\Psi_{10}=\Psi[\alpha,\beta]$ be the submatrix of the verification matrix $\Psi_S(A)$, where
$$\alpha=\{\{1, 4\}, \{1, 5\}, \{1, 6\}, \{2, 5\}, \{4, 5\}, \{3, 4\}, \{4, 6\},\{3, 5\}, \{4, 7\}, \{3, 6\}\}$$
and 
$$\beta=\{\{1,3\}, \{1,5\}, \{1,6\}, \{1,7\}, \{2,4\}, \{2,5\}, \{2,6\}, \{3,5\}, \{3,6\}, \{3,7\}\}.$$ 
Namely, 
\newcommand\scalemath[2]{\scalebox{#1}{\mbox{\ensuremath{\displaystyle #2}}}}
\[
\Psi_{10} =\left(
\scalemath{0.7}{
\begin{array}{cccccccccc}
 -a_{3,4} & -a_{4,5} & 0 & 0 & \frac{a_{2,3} a_{3,4}}{a_{1,4}} & 0 & 0 & 0 & 0 & 0 \\
 0 & a_{1,1}-a_{5,5} & -a_{5,6} & 0 & 0 & \frac{a_{2,3} a_{3,4}}{a_{1,4}} & 0 & 0 & 0 & 0 \\
 0 & -a_{5,6} & a_{1,1}-a_{6,6} & -a_{6,7} & 0 & 0 & \frac{a_{2,3} a_{3,4}}{a_{1,4}} & 0 & 0 & 0 \\
 0 & \frac{a_{2,3} a_{3,4}}{a_{1,4}} & 0 & 0 & -a_{4,5} & a_{2,2}-a_{5,5} & -a_{5,6} & a_{2,3} & 0 & 0 \\
 0 & a_{1,4} & 0 & 0 & 0 & 0 & 0 & a_{3,4} & 0 & 0 \\
 -a_{1,4} & 0 & 0 & 0 & a_{2,3} & 0 & 0 & -a_{4,5} & 0 & 0 \\
 0 & 0 & a_{1,4} & 0 & 0 & 0 & 0 & 0 & a_{3,4} & 0 \\
 0 & 0 & 0 & 0 & 0 & a_{2,3} & 0 & a_{3,3}-a_{5,5} & -a_{5,6} & 0 \\
 0 & 0 & 0 & a_{1,4} & 0 & 0 & 0 & 0 & 0 & a_{3,4} \\
0&  0 & 0 & 0 & 0 & 0 & a_{2,3} & -a_{5,6} & a_{3,3}-a_{6,6} & -a_{6,7} 
\end{array}}
\right).
\]
If $n=2$ we consider submatrix $\Psi'_{10}$ of $\Psi_{10}$ by deleting last two rows of $\Psi_{10}$ indexed by $\{4,7\}$ and $\{3,6\}$, and by deleting the columns corresponding to variables indexed by $\{1,7\}$ and $\{3,7\}$, namely 
$$\Psi'_{10}=
\tiny{\left(
\begin{array}{cccccccc}
 -a_{3,4} & -a_{4,5} & 0 & \frac{a_{2,3} a_{3,4}}{a_{1,4}} & 0 & 0 & 0 & 0  \\
 0 & a_{1,1}-a_{5,5} & -a_{5,6} & 0 & \frac{a_{2,3} a_{3,4}}{a_{1,4}} & 0 & 0 & 0 \\
 0 & -a_{5,6} & a_{1,1}-a_{6,6}  & 0 & 0 & \frac{a_{2,3} a_{3,4}}{a_{1,4}} & 0 & 0  \\
 0 & \frac{a_{2,3} a_{3,4}}{a_{1,4}} & 0  & -a_{4,5} & a_{2,2}-a_{5,5} & -a_{5,6} & a_{2,3} & 0  \\
 0 & a_{1,4} & 0 & 0 & 0 & 0 & a_{3,4} & 0 \\
 -a_{1,4} & 0 & 0 & a_{2,3} & 0 & 0 & -a_{4,5} & 0  \\
 0 & 0 & a_{1,4} & 0 & 0 & 0 & 0 & a_{3,4} \\
 0 & 0 & 0 & 0 & a_{2,3} & 0 & a_{3,3}-a_{5,5} & -a_{5,6} \\
\end{array}
\right)}.$$
Observe that $\det(\Psi_{10})=-8 a_{1, 4} a_{2, 3}^2 a_{3, 4}^3 a_{4, 5}^2 a_{5, 6} a_{6, 7} \neq 0$ and $\det(\Psi'_{10})=-4 a_{2,3}^2 a_{3,4}^3 a_{4,5}^2 a_{5,6}\neq 0$, so both matrices $\Psi_{10}$ and $\Psi'_{10}$ have full column rank. This implies that $$x_{1, 3}=x_{2,4}=x_{1,5}=x_{2,5}=x_{3,5}=0$$
in each of the three cases. Therefore $X\in \Sc(G_5^c)$, where  $G_5$ is the supergraph of $G$ on $V(G_5)=V(G)$ and
\begin{align*}
E(G_5)&=E(G)+\{\{i,j\}\colon 4\leq i<j\leq n+4\} + \{\{1,3\}, \{2,4\}, \{1,5\}, \{2,5\}, \{3,5\}\}
    \\
    &=\{\{i,j\}\colon 1\leq i<j\leq 5\} \cup \{\{i,j\}\colon 4\leq i<j\leq n+4\}.
\end{align*}
 We will now inductively prove $x_{1,m}=x_{2,m}=x_{3,m}=0$ for all $6 \leq m \leq n+4$.
 
  First we observe that $x_{j,6}=0$ for $j\in [3]$ using Lemma~\ref{lem:rule1} for $i=5$, $j\in[3]$ and $k=6$. 
     Denote $G_6=G_5+\{\{1,6\}, \{2,6\}, \{3,6\}\}$ and observe that $X\in \Sc(G_6^c)$.
    Suppose now we have proved for some $m \in \{6, ..., n+3\}$  that $X \in \Sc(G_{m}^c)$, where $G_{m}=G_{1}+\{\{1,t\}, \{2,t\}, \{3,t\} \colon 4\leq t \leq m\}$. 
    Let $i=m$, $j\in [3]$, $k=m+1$. Observe that again $N_G[i]\cap N_{G_{m}}[j]^c=\{k\}$, since $\{i, k \}$ is an edge in $G$ and $\{j,k \}$ is not an edge in $G_{m}.$ Furthermore $N_G[j]\cap N_{G_{m}}[i]^c=\emptyset$ since $i$ is a neighbour of all the vertices in graph $G_{m}$. Another use of Lemma~\ref{lem:rule1} implies that $ X\in \Sc(G_{m+1}^c)$, where $G_{m+1}=G_{m}+\{\{1,m+1\}, \{2,m+1\}, \{3,m+1\}\}$. 
Thus by induction $X\in\Sc(G_{n+4}^c)=\Sc(K_{n+4}^c)$ and so $X=O$, which implies that $A$ has the SSP.
\end{proof}

\subsection{Examples of graphs of girth four that are not SSP}\label{sec:girth-4-not-SSP}

In this subsection, we provide three examples of unicyclic graphs of girth four not in $\gssp$, for which the converse of Lemma~\ref{lem:unicyclic-not-in-gssp} is not valid. These examples differ in the number of pending paths, i.e.~on the number of degree three vertices on the cycle. This shows that the case when the girth of the graph is four differs from the results in Section~\ref{sec:girth=3}, where we characterized graphs of girth three. 

 \begin{center}
 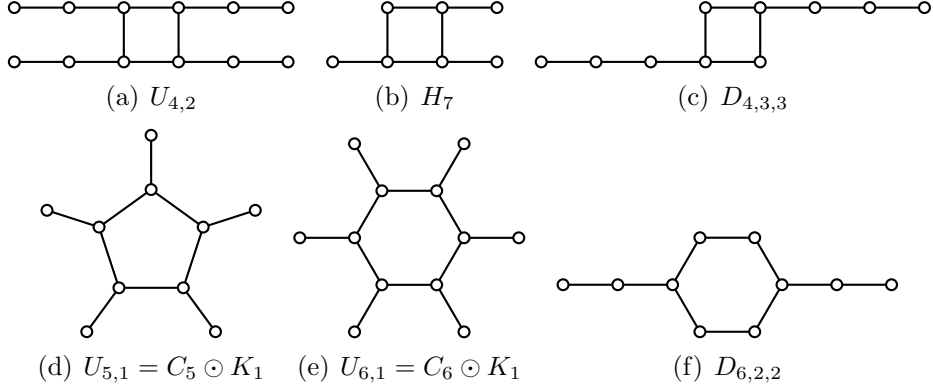
\begin{figure}[htb!]
 \centering
  \subfigure	[$U_{4,2}$]{
    \begin{tikzpicture}[scale=0.9,transform shape]
       \draw (2,1)--(2,0);
     \draw (3,1)--(3,0);
     \foreach \i in {0,1,2,3,4} {
           \draw (\i,0)--(\i+1,0);
           \draw (\i,1)--(\i+1,1);}
    \foreach \i in {0,1,2,3,4,5} {      
            \node[fill=white] at (\i,0) {}; 
            \node[fill=white] at (\i,1) {}; 
        }
    \end{tikzpicture}  \label{g4:U4m}}
     \subfigure	[$H_{7}$]{
    \begin{tikzpicture}[scale=0.9,transform shape]
    \draw (2,1)--(2,0);
     \draw (3,1)--(3,0);
     \foreach \i in {1,2,3} {
           \draw (\i,0)--(\i+1,0);
           \node[fill=white] at (\i,0) {}; 
            }
    \foreach \i in {2,3} {
           \draw (\i,1)--(\i+1,1);
            \node[fill=white] at (\i,1) {}; }
    \node[fill=white] at (4,1) {}; 
    \node[fill=white] at (4,0) {}; 
    \end{tikzpicture}   \label{g4:H7}}
   \subfigure	[$D_{4,3,3}$]{
    \begin{tikzpicture}[scale=0.9,transform shape]
       \draw (2,1)--(2,0);
     \draw (3,1)--(3,0);
     \foreach \i in {-1,0,1,2} {
           \draw (\i,0)--(\i+1,0);
           \node[fill=white] at (\i,0) {}; 
            }
    \foreach \i in {2,3,4,5} {
           \draw (\i,1)--(\i+1,1);
            \node[fill=white] at (\i,1) {}; }
    \node[fill=white] at (6,1) {}; 
    \node[fill=white] at (3,0) {}; 
    \end{tikzpicture}   \label{g4:D433}} 
     \subfigure	[$U_{5,1}=C_5\odot K_1$]{
    \begin{tikzpicture}[scale=0.9,transform shape]
     \foreach \i in {1,2,3,4,5} {
           \draw (\i*360/5+90:1)--(\i*360/5+90+360/5:1);
           \draw (\i*360/5+90:1)--(\i*360/5+90:2);}
        \foreach \i in {1,2,3,4,5} {   
           \node[fill=white] at (\i*360/5+90:1) {}; 
           \node[fill=white] at (\i*360/5+90:2) {}; 
        }
    \end{tikzpicture}  \label{g4:U5m}}
    \subfigure	[$U_{6,1}=C_6\odot K_1$]{
    \begin{tikzpicture}[scale=0.9,transform shape]
     \foreach \i in {0,60,120,...,300} {
           \draw (\i:1)--(\i+60:1);
           \draw (\i:1)--(\i:2);}
        \foreach \i in {0,60,120,...,300} {   
           \node[fill=white] at (\i:1) {}; 
           \node[fill=white] at (\i:2) {}; 
        }
    \end{tikzpicture}  \label{g4:U6m}}
     \subfigure	[$D_{6,2,2}$]{
    \begin{tikzpicture}[scale=0.9,transform shape]
      \draw (0:1)--(0:3);
      \draw (180:1)--(180:3);
      \foreach \i in {0,60,120,...,300} {
           \draw (\i:1)--(\i+60:1);
           }
        \foreach \i in {0,60,120,...,300} {   
           \node[fill=white] at (\i:1) {}; 
        }
        \foreach \i in {2,3} {   
           \node[fill=white] at (0:\i) {}; 
           \node[fill=white] at (180:\i) {}; 
        }
    \end{tikzpicture}  \label{g4:D622}}
    \caption{Five examples of unicyclic graphs of girth four not in $\gssp$, see Examples~\ref{ex:4cycle+4paths}, \ref{ex:4cycle+3paths} and \ref{ex:even-cycle+2paths}.}\label{fig:g4-not-ssp}
\end{figure}
\end{center}

\begin{lemma}\label{lem:adding-pending-paths}
    Let $G$ be a graph, $G\notin\gssp$, and let $H$ be obtained by adding pending paths of the same size from each of the vertices in $G$. Then $H\notin \gssp$. 
\end{lemma}
\begin{proof}
Suppose $G\notin \gssp$ has $n$ vertices and let $A \in S(G)$ and $O \neq X \in \Sc(G^c)$ be matrices such that $[A,X]=O$. Suppose $H$ is obtained from $G$ by adding pending paths of length $m$ from each of the vertices of $G$. Observe that $|H|=(m+1)n$.
Define now $B=[B_{i,j}]_{i,j=1,\ldots,m+1}\in\calS(H)$, where $B_{i,j}\in\bR^{n\times n}$, as
$$B_{i,j}=\begin{cases}
    A,& \text{if} \ i=j=1,\\
    I_n, & \text{if}\ |i-j|=1,\\
    O,& \text{otherwise.}
\end{cases}$$ 
Define $Y=X\oplus\cdots\oplus X\in\Sc(H^c)$ and observe that 
$[B,Y]=0$ if and only if $[A,X]=O$. Now, since $X\neq O$ and $A$ and $X$ commute, it follows that $Y \neq 0$ and also $B$ and $Y$ commute, and hence $H \notin \gssp$.
    \end{proof}

\begin{remark}
 Note that the converse of Lemma~\ref{lem:adding-pending-paths} is not true in general. For example, let $G=K_{1,3}\in \gssp$ and let  $H$ be a corona graph $K_{1,3}\odot K_1$. In this case, $H$ is obtained by adding a pending leaf from each of the four vertices of $K_{1,3}$. Note that $H$ is a generalized star with four arms of lengths $2,2,2,1$ and hence by~\cite[Theorem~4.3]{Lin20SSPgraph} it is not an SSP graph.
\end{remark}
    
\begin{example}\label{ex:4cycle+4paths}
Suppose $n\geq 4$. In [Example 2.1,~\cite{Lin20SSPgraph}] it is proved that adjacency matrix $A={\rm Adj}(C_n) \in S(C_n)$ does not have the SSP since it commutes with matrix $X=(x_{i,j})$ where $x_{i,j}=1$ if and only if $\{i,j\} \in E(C_n^c)$ and $i\ne j$. 
Let us denote by $U_{n,m}$ the graph obtained from the cycle $C_n$ by adding pending paths $P_m$ from all of its vertices, see Figures~\ref{g4:U4m},~\ref{g4:U5m} and~\ref{g4:U6m} for three particular examples of graphs $U_{4,2}$, $U_{5,1}$ and $U_{6,1}$.
If we take $G=C_n$ in Lemma~\ref{lem:adding-pending-paths}, then it follows that $U_{n,m} \notin \gssp$ for all $n\geq 4$ and $m\in \bN_0$.
\end{example}

\begin{example}\label{ex:4cycle+3paths}
Let $H_7$ be the graph presented on Figure~\ref{g4:H7} and let us define 
$$A=\left(
\begin{array}{ccccccc}
 0 & -1 & 0 & 1 & 0 & 0 & 0 \\
 -1 & 0 & -1 & 0 & -\sqrt{2} & 0 & 0 \\
 0 & -1 & 0 & -1 & 0 & 2 & 0 \\
 1 & 0 & -1 & 0 & 0 & 0 & -\sqrt{2} \\
 0 & -\sqrt{2} & 0 & 0 & \sqrt{2} & 0 & 0 \\
 0 & 0 & 2 & 0 & 0 & 0 & 0 \\
 0 & 0 & 0 & -\sqrt{2} & 0 & 0 & -\sqrt{2} \\
\end{array}
\right)\in \calS(H_7) $$
and
$$X=\left(
\begin{array}{ccccccc}
  0 & 0 & \sqrt{2} & 0 & -1 & 0 & -1 \\
 0 & 0 & 0 & 0 & 0 & -\sqrt{2} & -\sqrt{2} \\
 \sqrt{2} & 0 & 0 & 0 & 1 & 0 & -1 \\
 0 & 0 & 0 & 0 & -\sqrt{2} & \sqrt{2} & 0 \\
 -1 & 0 & 1 & -\sqrt{2} & 0 & 0 & 0 \\
 0 & -\sqrt{2} & 0 & \sqrt{2} & 0 & 0 & 0 \\
 -1 & -\sqrt{2} & -1 & 0 & 0 & 0 & 0 \\
\end{array}
\right)\in\Sc(H_7^c)$$
It is a straightforward calculation that $A\circ X=[A,X]=O$,
which shows that matrix $A\in \calS(H_7)$ does not have the SSP. Therefore, ${H_7}$ is not an SSP graph. 
\end{example}

\begin{example}\label{ex:even-cycle+2paths}
For $n\geq 2$ and $m\in \bN_0$, let $D_{2n,m,m}$ be graph obtained by joining an even cycle $C_{2n}$ by two bridges from two diametrically opposite vertices to two leaves of two distinct copies of path $P_m$. Graph $D_{2n,m,m}$ is a unicyclic graph of girth $2n$, and see $D_{4,3,3}$ and $D_{6,2,2}$ on Figures~\ref{g4:D433} and~\ref{g4:D622} for two particular examples.

So defined graph has $|D_{2n,m,m}|=2n+2m$ vertices and let $V(D_{2n,m,m})=[2n+2m]$ and 
\begin{align*}
 E(D_{2n,m,m})=&\{\{1,2n-1\},\{2,2n-1\}\{2n-3,2n\}\}\cup \\
 &\qquad \qquad \{\{i,i+2\}\colon i\in [2n+2m-2]\setminus \{2n-3\}\}.
 \end{align*}

Let $A={\rm A}(D_{2n,m,m})$ be the adjacency matrix of  $D_{2n,m,m}$. Again, $A=[B_{i,j}]_{i,j=1,\ldots,n+m}$ can be seen as a block matrix with $n+m$ blocks $B_{i,j}$ of order $2$, where $$ B_1=\left(
\begin{array}{cc}
 1 & 0 \\
 1 & 0 \\
\end{array}\right), B_2=\left(
\begin{array}{cc}
 0 & 1 \\
 0 & 1 \\
\end{array}\right) \; \text{ and }\; B_{i,j}=\begin{cases}
    B_1,& \text{if } (i,j)=(1,n),\\
    B_2,& \text{if } (i,j)=(n-1,n),\\
    I_2, & \text{if } j=i+1 \text{ and } j\ne n\\
     O,& \text{otherwise, \ }
\end{cases}$$ and $I_2$ is the identity matrix of order $2$.
Let $X=[X_{i,j}]_{i,j=1,\ldots,n+m}\in \Sc(D_{2n,m,m}^c)$ be also a block matrix with $n+m$ blocks $X_{i,j}$ of size $2\times 2$, and where $$X_{i,j}=\begin{cases}
    A(K_2), & \text{for } j=i, j>n-1 \ \text{or} \ i+j=n,\\
     0,& \text{for all the other } i, j,\\
\end{cases}$$
and $A(K_2)$ denotes the adjacency matrix of $K_2$. 
Since $B_1A(K_2)=A(K_2)B_2$, it is straightforward to check that $[A,X]=O$, and so the adjacency matrix $A$ does not have the SSP. Therefore, $D_{2n,m,m}\notin\gssp$.
\end{example}

\section{Unicyclic graphs of girth five}\label{sec:girth=5}

The smallest graph of girth five is cycle $C_5$, known not to be an SSP graph, see~\cite[Example~2.1]{Lin20SSPgraph}. Moreover, its corona graph $U_{5,1}$ is not an SSP graph as well; see Example~\ref{ex:4cycle+4paths}. We prove in this chapter that all tadpole graphs $T_{5,n}$ are SSP graphs.

First, let us relate the verification matrices of a graph and its subgraph. Let us consider a connected graph $H=(V(H),E(H))$ on vertices $V(H)=[n+1]$ and its induced subgraph $G=H[\{1,\ldots,n\}]
$. 
Let \begin{equation*}
B=\begin{pmatrix}
    A&\bfb\\
    \bfb\trans & \beta
\end{pmatrix}\in \calS(H) \text{ and } Y=\begin{pmatrix}
    X&\bfy\\
    \bfy\trans & 0
\end{pmatrix}\in \Sc(H^c)
\end{equation*}
where $A=B[[n]]\in \calS(G)$, 
$\beta \in \bR$, and $\bfb
, \bfy
\in \bR^n$ such that $\bfb \circ \bfy ={\bf 0}$.
The condition $[B,Y]=O$ is equivalent to the system of linear equations
\begin{align}
    [A,X]+\bfb \bfy\trans-\bfy\bfb\trans&=O \label{eq3}\\ 
    X\bfb+(\beta I_n-A)\bfy&={\bf 0}.\nonumber
\end{align}
We can rewrite this system to
\begin{align}
    AXI_n-I_nXA+\bfb \bfy\trans I_n-I_n\bfy\bfb\trans&=O\nonumber\\
    I_nX\bfb+(\beta I_n-A)\bfy &={\bf 0},\nonumber
\end{align}
and after vectorization, we obtain
\begin{align*}
    (I_n\otimes A-A\otimes I_n)\vecop(X)+(I_n\otimes\bfb- \bfb\otimes I_n)\bfy&={\bf 0}\\
    (\bfb\trans \otimes I_n)\vecop(X)+(\beta I_n-A)\bfy&={\bf 0},
\end{align*}
which can be rewritten as 
\begin{equation}\label{eq:mtx-form}
    \begin{pmatrix}
        I_n\otimes A-A\otimes I_n& I_n\otimes\bfb- \bfb\otimes I_n\\
    \bfb\trans \otimes I_n &\beta I_n-A
    \end{pmatrix}\begin{pmatrix}
        \vecop(X)\\
        \bfy
    \end{pmatrix}={\bf 0}.
\end{equation}
Note that $\vecop(X)$ has $n^2$ components, and since $X$ is symmetric, it contains in addition to $|E(G)|+n$ zeros also each variable $x_{ij}$ twice. Also, vector $\bfy$ may contain some zeros. After eliminating possible zeros in $\vecop(X)$ and $\bfy$ and merging the same variables, our system~\eqref{eq:mtx-form} becomes
\begin{align}\label{eq:lin-system-Psi}
    \begin{pmatrix}
        A_1 & B_1\\
        C_1 &D_1
    \end{pmatrix}
    \begin{pmatrix}
        \bfx_G\\
        \bfx'
    \end{pmatrix}={\bf 0},
\end{align}
where $\bfx_G=(x_{ij})$ for $i\neq j$ and $\{i,j\}\in E(G^c)$, the variables in $\bfx'$ are ordered in ascending lexicographic order, and $\bfx'=(x_{i,n+1})=\supp(\bfy)$, where $
\{i,n+1\}\notin E(H)$. Note that for some graphs, it will happen that some variables of $\vecop(X)$ or $\bfy$ will be immediately forced to be zeros by some of the equations in $[A,X]=O$, i.e. in the case of pending paths, see Lemma~\ref{lem:corofcor}. In this case, we put them as zeros first, so they will not appear in $\bfx_G$ or $\bfx'$. Moreover, note that the number of columns of $B_1$ equals the size of $\bfx'$.

Now observe that  $A_1=\Psi_{S}(A)$ is the SSP verification matrix of $A$ and  \begin{equation}\label{eq:PsiB}\Psi_B=\begin{pmatrix}
        A_1 & B_1\\
        B_2 &A_2
    \end{pmatrix}\end{equation} is the permuted SSP verification matrix of $B$. So we have the following.

\begin{lemma}\label{lem:verification-matrix-subgraph}
    If a graph $H$ is obtained from a graph $G$ as  $V(H)=V(G)\cup \{w\}$ and $E(G)\subseteq E(H)$, then the SSP verification matrix for any  $A\in \calS(G)$ is the submatrix of the SSP verification matrix for any $B\in \calS(H)$.

    Moreover, if $w$ is a leaf in $H$, then the SSP verification matrix of $B$ is permutationally similar to matrix 
    $$\Psi_B=\begin{pmatrix}
         \Psi_S(A) & B_1\\
         B_2 & A_2
     \end{pmatrix},$$ where  $\Psi_S(A)$ is the SSP verification matrix of $A$ and 
     $ B_1=- a_{n,n+1}\begin{pmatrix} O\\   I\end{pmatrix}$.
\end{lemma}    
\begin{proof}
   The first part of the lemma follows by equation~\eqref{eq:PsiB}. Now, without loss of generality, let $w=n+1$ be a leaf attached to vertex $v=n$ in graph $G$ with $V(G)=[n]$.  Using notation above,~\eqref{eq3} now tells us that system $[A,X]=O$ differs from system $[A,X]=\bfy\bfb\trans-\bfb\bfy\trans$ only in equations that belong to positions 
   $$\alpha=\{\{i,n\}: i\in[n-1], x_{i,n+1}\neq 0\}.$$
   Since $$\bfb=(0,0,...,0, a_{n,n+1})\trans \text{ and } \bfy=(x_{1,n+1}, x_{2,n+1},...,x_{n-1,n+1}, 0)\trans,$$
   observe that $B_1[\alpha^c,\colon]=O$ and  $B_1[\alpha,\colon]=-a_{n,n+1}I$.
\end{proof}

In the proof of the following theorem, we will also investigate the structure of submatrix $B_1$ in~\eqref{eq:PsiB} in the case when a leaf $w=n+1$ in $H$ is attached to a leaf $v=n\in V(G)$. 

\begin{theorem}\label{thm:girth-5}
Tadpole graph $T_{5,n}$ is in $\gssp$ for all $n\geq 1$.

\end{theorem}
\begin{proof}
Observe first that the case $n=1$ is covered in Example~\ref{ex:t51ssp}, where we proved that $T_{5,1}\in \gssp$.

Assume now $n\geq 2$ and let us label the vertices of $T_{5,n}$ as $V(T_{5,n})=[n+5]$ and 
$$E(T_{5,n})=\{\{1,5\}\} \cup \{\{i,i+1\}\colon i\in [n+4]\}.$$ Let us choose arbitrary $A=(a_{i,j})\in \calS(T_{5,n})$ and $X\in \Sc(T_{5,n}^c)$.
Note first that $x_{i,j}=0$, $5\leq i\leq j \leq n+5$, because of Lemma~\ref{lem:corofcor}.
First, we claim that $x_{1,6}=x_{4,6}=0$.

Using Lemma~\ref{lem:verification-matrix-subgraph} for $n-1$ times, we see that the permuted verification matrix $\Psi_A$ of $A$ is of the form $$\Psi_A=\begin{pmatrix}
    \Psi' & B_1\\
    B_2 & A_2
\end{pmatrix},$$
where $\Psi'$ is a permuted verification matrix of $A'\in \calS(T_{5,1})$, see Example~\ref{ex:t51ssp}, and $ B_1=- a_{6,7}\begin{pmatrix} O & O\\   I & O\end{pmatrix}$. Note that
$\Psi'_\alpha$ in~\eqref{eq:invertible-submatrix-of-verification-matrix-T51}
is an invertible submatrix of $\Psi'$ 
for every $A'\in \calS(T_{5,1})$, and $\Psi'_{\alpha}$ 
contains all the columns of $\Psi'$. Let us abbreviate $\Psi_\alpha:=\Psi'_{\alpha}$. Let $\beta=\{\{1,7\},\{2,7\},\{3,7\},\{4,7\}\}$. Now the matrix
$\Psi_A$ can be further decomposed as
    \begin{align}\label{eq:lin-system-Psi-invertible-part}
    \Psi_A= \begin{pmatrix}
        \Psi_{\alpha} & B_{\alpha}\\
        \Psi_N & B_N\\
        B_2 & A_2
    \end{pmatrix},
\end{align}
where the last four columns of matrices $\Psi_{\alpha}$, $\Psi_N$ and $B_2$ correspond to variables $x_{1,6}, x_{2,6}, x_{3,6}$ and $x_{4,6}$. Moreover,  $B_{\alpha}=\begin{pmatrix} B_{\alpha,\beta} & O
\end{pmatrix}$, where
$$B_{\alpha,\beta}=B_{\alpha}[\colon,\beta ]=\begin{pmatrix}
 0 & 0 & 0 & 0 \\
 0 & 0 & 0 & 0 \\
 0 & 0 & 0 & 0 \\
 0 & 0 & 0 & 0 \\
 0 & 0 & 0 & 0 \\
 -a_{6,7} & 0 & 0 & 0 \\
 0 & -a_{6,7} & 0 & 0 \\
 0 & 0 & 0 & -a_{6,7} \\
 0 & 0 & 0 & 0 \end{pmatrix},$$ and the four columns of $B_{\alpha,\beta}$ correspond to variables $x_{1,7}, x_{2,7}, x_{3,7}$ and $x_{4,7}$.

 Let us show that $x_{1,6}=0$. If we set \begin{align*}
\bfx_6&=(x_{1,3}, x_{1,4}, x_{2,4}, x_{3,4}, x_{3,5}, x_{1,6}, x_{2,6}, x_{3,6}, x_{4,6})\trans \, \text{ and }\\
\, \bfx'&=(x_{1,7}, x_{2,7}, x_{3,7}, x_{4,7},\bfx_8\trans)\trans,\end{align*}
where all other variables $\{x_{i,j}\colon 1\leq i\leq 4, 8\leq j\leq n+5\}$ are collected in lexicographically ascending order in $\bfx_8$.
From decomposition~\eqref{eq:lin-system-Psi-invertible-part} and the system
$    \Psi
    \begin{pmatrix}
        \bfx_6\\
        \bfx'
    \end{pmatrix}={\bf 0}$, it follows that $\Psi_{\alpha} \bfx_6+B_{\alpha} \bfx'=\bf 0$, and the invertibility of $\Psi_{\alpha}$ implies 
$\bfx_6=-\Psi_{\alpha}^{-1} B_{\alpha} \bfx'$.    
To prove that $x_{1,6}=0$, we have to prove that the sixth entry of $\bfx_6$ is equal to $0$, or equivalently  that the sixth row of $\Psi_{\alpha}^{-1}B_{\alpha}$ or $\Psi_{\alpha}^{-1}B_{\alpha,\beta}$ is ${\bf 0}$. 
Note that the latter is equal to 
$$(\Psi_{\alpha}^{-1}B_{\alpha,\beta})[\{6\},:]=\begin{pmatrix}
    \mu_{6,6} M_{6,6} & \mu_{7,6} M_{7,6} & 0 & \mu_{8,6} M_{8,6}
\end{pmatrix}$$
for some suitable $\mu_{i,j}\in \bR$, and where $M_{i,j}$ denotes the minor of $\Psi_{\alpha}$ obtained by deleting the $i$th row and the $j$th column of $\Psi_{\alpha}$.
Hence it is enough to show that $M_{6,6}=M_{7,6}=M_{8,6}=0$.
The first of them is equal to
\begin{align*}
M_{6,6}&=\det\begin{pmatrix}
    a_{1,2} & 0 & -a_{3,4} & 0 & 0 & 0 & 0 & 0 \\
 0 & -a_{4,5} & 0 & a_{1,2} & 0  & 0 & 0 & 0 \\
 -a_{2,3} & 0 & 0 & a_{1,5} & 0 & 0 & 0 & 0 \\
 0 & 0 & a_{2,3} & 0 & -a_{4,5} & 0 & 0 & 0 \\
 0 & -a_{1,5} & 0 & 0 & a_{3,4} & 0 & 0 & -a_{5,6} \\
 0 & 0 & 0 & -a_{5,6} & 0  & a_{2,2}-a_{6,6} & a_{2,3} & 0 \\
 0 & 0 & 0 & 0 & 0 & 0 & a_{3,4} & a_{4,4}-a_{6,6} \\
 0 & 0 & 0 & 0 & 0  & 0 & 0 & a_{4,5}
\end{pmatrix}=\\
&=(a_{2,2}-a_{6,6})a_{3,4}a_{4,5} \det\left(
\begin{array}{ccccc}
 a_{1,2} & 0 & -a_{3,4} & 0 & 0  \\
 0 & -a_{4,5} & 0 & a_{1,2} & 0  \\
 -a_{2,3} & 0 & 0 & a_{1,5} & 0   \\
 0 & 0 & a_{2,3} & 0 & -a_{4,5} \\
 0 & -a_{1,5} & 0 & 0 & a_{3,4}  
 \end{array}
\right)=0.
\end{align*}
Similar arguments can be used to show that $M_{7,6}=M_{8,6}=0$, which implies $x_{1,6}=0$.
From the equation $[A,X]_{5,6}=0$ observe that $x_{4,6}=0$ as well. 

Now let us substitute $x_{1,6}=x_{4,6}=x_{i,j}=0$, $5\leq i \leq j \leq n+5$ and let vector $\bfx$ of the remaining variables, which are corresponding to $E(G^c)\setminus \{\{i,j\}\colon 5\leq i \leq j \leq n+5\}\setminus \{x_{1,6},x_{4,6}\}$, be decomposed as 
$$\bfx=\begin{pmatrix}
    x_{1,3}\\
    \bfx_1\\
    \bfx_8
\end{pmatrix},$$
where
$$
\bfx_1:=(x_{2,5},x_{2,4},x_{3,5},x_{1,4},x_{2,6},x_{3,6},x_{1,7},x_{2,7},x_{3,7},x_{4,7})\trans$$
and let all other variables $\{x_{i,j}\colon 1\leq i\leq 4, 8\leq j\leq n+5\}$ be collected in lexicographically ascending order in $\bfx_8$.

Moreover, we abuse the notation and denote the new verification matrix of $[A,X]=O$ again by $\Psi$, where the columns are ordered as in $\bfx$. 
Let us decompose the indices of rows  of $\Psi$ to
\begin{align*}
\alpha_1&:=\{\{1,2\},\{2,3\},\{3,4\},\{4,5\},\{2,5\},\{3,5\},\{1,6\},\{2,6\},\{3,6\},\{4,6\}\},\\
\alpha_2&:=\{\{1,3\},\{1,4\},\{2,4\},\{5,7\}\}  \text{ and } \\
\alpha_3&:=(\alpha_1\cup \alpha_2)^c.
\end{align*}
If $\beta_1$ is the subset of columns corresponding to variables in $\bfx_1$, then observe that in the chosen order of variables and equations the matrix  $\Psi_0:=\Psi[\alpha_1, \beta_1]$ is equal to 
\newcommand\scalemath[2]{\scalebox{#1}{\mbox{\ensuremath{\displaystyle #2}}}}
$$\Psi_0=\left(
\scalemath{0.8}{
\begin{array}{cccccccccccc}
a_{1,5} & 0 & 0 & 0 & 0 & 0 & 0 & 0 & 0 & 0 \\
 0 & -a_{3,4} & 0 & 0 & 0 & 0 & 0 & 0 & 0 & 0 \\
 0 & a_{2,3} & -a_{4,5} & 0 & 0 & 0 & 0 & 0 & 0 & 0 \\
 0 & 0 & a_{3,4} & -a_{1,5} & 0 & 0 & 0 & 0 & 0 & 0 \\
 a_{2,2}-a_{5,5} & -a_{4,5} & a_{2,3} & 0 & -a_{5,6} & 0 & 0 & 0 & 0
   & 0 \\
 a_{2,3} & 0 & a_{3,3}-a_{5,5} & 0 & 0 & -a_{5,6} & 0 & 0 & 0 & 0 \\
 0 & 0 & 0 & 0 & a_{1,2} & 0 & -a_{6,7} & 0 & 0 & 0 \\
 -a_{5,6} & 0 & 0 & 0 & a_{2,2}-a_{6,6} & a_{2,3} & 0 & -a_{6,7} & 0
   & 0 \\
 0 & 0 & -a_{5,6} & 0 & a_{2,3} & a_{3,3}-a_{6,6} & 0 & 0 & -a_{6,7}
   & 0 \\
 0 & 0 & 0 & 0 & 0 & a_{3,4} & 0 & 0 & 0 & -a_{6,7} \\
\end{array}}
\right)$$ 
and is invertible. Since  $\Psi[\alpha_1, \overline{\beta_1}\setminus\{x_{1,3}\}]=O$, the equations of $\Psi \bfx ={\bf 0}$ on positions $\alpha_1$ imply $x_{1,3} \bfv+\Psi_0 \bfx_1 + O \bfx_8=\bf 0$, where $\bfv=\Psi[\alpha_1,\{1,3\}]$. Therefore, 
\begin{equation}\label{eq:x_1-by-v}
    \bfx_1=-x_{1,3} \Psi_0^{-1} \bfv,
\end{equation} and hence we expressed all variables in $\bfx_1$ by $x_{1,3}$ and entries of $A$. 

Suppose that $x_{1,3}\ne 0$ and observe that $\Psi[\alpha_2,  \overline{\beta_1}\setminus\{x_{1,3}\}]=O$ and the four equations in $\frac{1}{x_{1,3}} [A,X]|_{\alpha_2}=O$  are linear in variables $\{a_{3,3}, a_{4,4},a_{2,2}, a_{5,5}\}$. The corresponding $4\times 4$ matrix of this system is equal to $$S=\begin{pmatrix}
 1 & 0 & 0 & 0 \\
 0 & \frac{a_{1,2} a_{2,3}}{a_{1,5} a_{4,5}} & 0 & 0 \\
 0 & \frac{a_{1,2}}{a_{3,4}} & -\frac{a_{1,2}}{a_{3,4}} & 0 \\
 -\frac{a_{1,2} a_{2,3}}{a_{5,6} a_{6,7}} & 0 & -\frac{a_{1,2} a_{2,3}}{a_{5,6} a_{6,7}} & \frac{2 a_{1,2} a_{2,3}}{a_{5,6} a_{6,7}} 
\end{pmatrix},$$
and so invertible. Hence 
\begin{equation}\label{eq:four-diagonal-entries-of-A}
    \begin{pmatrix}
    a_{3,3}\\ a_{4,4}\\a_{2,2}\\a_{5,5}
\end{pmatrix}=S^{-1} \begin{pmatrix}
    -a_{1,1}-\frac{a_{1,2} a_{2,3} \left(a_{1,5}^2-a_{3,4}^2\right)}{a_{1,5} a_{3,4} a_{4,5}} \\
 -\frac{a_{1,2}^2}{a_{3,4}}-\frac{a_{1,1} a_{2,3} a_{1,2}}{a_{1,5} a_{4,5}}+a_{3,4} \\
 \frac{a_{2,3} \left(a_{4,5}^2-a_{1,2}^2\right)}{a_{1,5} a_{4,5}} \\
 \frac{a_{1,2}^2 \left(a_{4,5}^2-a_{2,3}^2\right) a_{1,5}^2+\left(a_{1,5}^2-a_{2,3}^2\right) a_{3,4}^2 a_{4,5}^2}{a_{1,5} a_{3,4} a_{4,5} a_{5,6} a_{6,7}}
\end{pmatrix},
\end{equation}
so we expressed four diagonal entries $a_{3,3}, a_{4,4}, a_{2,2}$ and $a_{5,5}$ of $A$ by other entries of $A$. 
In the case $n=2$, observe now that $$[A,X]_{1,7}=-\frac{a_{1,2}a_{2,3}a_{5,6}x_{1,3}}{a_{6,7}a_{1,5}} \ne 0,$$
a contradiction with $[A,X]=O$.
If $n\geq 3$, then after the substitution~\eqref{eq:four-diagonal-entries-of-A}, we obtain 
$$[A,X]_{1,7}=-\frac{a_{1,2}a_{2,3}a_{5,6}x_{1,3}}{a_{6,7}a_{1,5}}-a_{7,8}x_{1,8} \text{ and } [A,X]_{4,7}=-\frac{a_{1,2}a_{2,3}a_{5,6}x_{1,3}}{a_{6,7}a_{4,5}}-a_{7,8}x_{4,8}.$$ Since both entries have to be equal to zero, we have $$x_{1,8}=-\frac{a_{1,2}a_{2,3}a_{5,6}x_{1,3}}{a_{6,7}a_{1,5}a_{7,8}} \ \text{and}\ x_{4,8}=-\frac{a_{1,2}a_{2,3}a_{5,6}x_{1,3}}{a_{6,7}a_{4,5}a_{7,8}}.$$ After substitution we obtain $$[A,X]_{5,8}=-2\frac{a_{1,2}a_{2,3}a_{5,6}x_{1,3}}{a_{6,7}a_{7,8}}\ne 0,$$
which again contradicts $[A,X]=O$.
Therefore $x_{1,3}=0$, and so $\bfx_1={\bf 0}$ by~\eqref{eq:x_1-by-v}. Now it follows from equations in positions $\{1,k\},\{2,k\}, \{3,k\}, \{4,k\}$, $7\leq k\leq n+4$, that  $\bfx_8={\bf 0}$, so $\bfx={\bf 0}$ and therefore $X=O$. Hence, any $A\in \calS(T_{5,n})$ has the SSP and so $T_{5,n}\in \gssp$.
\end{proof}

\section{Further examples and questions}\label{sec:girth=6}

In the last three sections, we showed that all tadpole graphs of girth at most five are in $\gssp$, but we could not find any other unicyclic graphs of girth four or five that would be in $\gssp$ as well. 
Even more, if more than one vertex on the 4-cycle has degree three, we have shown examples of graphs not in the $\gssp$. We also provided examples of unicyclic graphs of larger girths not in $\gssp$. Therefore, the following question arises naturally.

\begin{question}
    Are there any girth four or five unicyclic graphs except $T_{4,n}$ or $T_{5,n}$ in $\gssp$?
\end{question}

We saw in Examples~\ref{ex:4cycle+4paths} and~\ref{ex:even-cycle+2paths} that $U_{n,m},D_{2n,m,m}\notin\gssp$, in particular, $U_{6,m},D_{6,m,m}\notin \gssp$. In the next example, we also show that $T_{6,1}\notin \gssp$. This proves that not all tadpole graphs are the SSP graphs, hence Theorems~\ref{t4nssp} and~\ref{thm:girth-5} cannot be extended to unicyclic graphs of girth more than five.

\begin{example}
Let us define 
$$A=\left(
\begin{array}{ccccccc}
 1 & -1 & 0 & 0 & 0 & 1 & 0 \\
 -1 & 2 & 1 & 0 & 0 & 0 & 0 \\
 0 & 1 & 0 & 1 & 0 & 0 & 0 \\
 0 & 0 & 1 & 2 & -1 & 0 & 0 \\
 0 & 0 & 0 & -1 & 1 & 1 & 0 \\
 1 & 0 & 0 & 0 & 1 & 1 & 1 \\
 0 & 0 & 0 & 0 & 0 & 1 & 1 \\
\end{array}
\right)\in \calS(T_{6,1}).$$
It is a straightforward calculation that $A$ commutes with 
$$X=\left(
\begin{array}{ccccccc}
 0 & 0 & 1 & 0 & -1 & 0 & 0 \\
 0 & 0 & 0 & -1 & 0 & 1 & 0 \\
 1 & 0 & 0 & 0 & 1 & -1 & 1 \\
 0 & -1 & 0 & 0 & 0 & 1 & 0 \\
 -1 & 0 & 1 & 0 & 0 & 0 & 0 \\
 0 & 1 & -1 & 1 & 0 & 0 & 0 \\
 0 & 0 & 1 & 0 & 0 & 0 & 0 \\
\end{array}
\right),$$
which shows that matrix $A\in \calS(T_{6,1})$ does not have the SSP. Therefore, $T_{6,1}$ is not an SSP graph. 
\end{example}

\begin{question}
    To the best of our knowledge, there is no known unicyclic graph of girth at least $6$ in $\gssp$. Are there any unicyclic graphs of girth $g\geq 6$ in $\gssp$? Which tadpole graphs $T_{m,n}$, $m\geq 6$, are in $\gssp$?
\end{question}

 \begin{question}\label{q:leaf-to-leaf} 
 Let graph $H$ be obtained from a graph $G$ as $V(H)=V(G)\cup \{w\}$ and $E(G)\subseteq E(H)$. 
 
First note that the SSP of the graph $G$ does not imply the SSP of the graph $H$, for example graph $K_{1,3} \in \gssp$ but $K_{1,4}\notin \gssp$,  \cite[Theorem 4.3]{Lin20SSPgraph}. 
 
   We saw that in the case of trees, graphs of girth three, and the tadpole graphs $T_{4,n}$ and $T_{5,n}$, the lengths of the paths do not affect the SSP of the graph. 
     Is it true that for any $G\in \gssp$ having a leaf $v\in V(G)$, every graph $H$, obtained as $$V(H)=V(G)\cup \{w\} \text{ and } E(H)=E(G) \cup \{\{v,w\}\}$$
  is also in the set $\gssp$?
   \end{question}

\bibliographystyle{plain}
\bibliography{bibliography}
\end{document}